\documentclass[10pt]{amsart}
\usepackage{amssymb}
\usepackage{epsfig}
\usepackage{url}
\usepackage[T1]{fontenc}
\usepackage{color}
\usepackage{tikz}
\usepackage{pgf}
\usetikzlibrary{arrows}
\usetikzlibrary{automata,positioning}
\theoremstyle{plain}

\newtheorem{thm}{Theorem}[section]
\newtheorem{cor}[thm]{Corollary}
\newtheorem{lem}[thm]{Lemma}

\newtheorem{prob}[thm]{Problem}
\newtheorem{rem}[thm]{Remark}
\newtheorem{ques}[thm]{Question}
\newtheorem{conj}[thm]{Conjecture}

\def\bbb{\mathbb}
\def\op{\operatorname}
\renewcommand{\phi}{\varphi}

\newcommand{\N}{\bbb{N}}
\newcommand{\Z}{\bbb{Z}}

\newcommand{\F}{\bbb{F}}

\newcommand{\C}{\bbb{C}}
\newcommand{\eps}{\varepsilon}

\newcommand{\Mod}[1]{\ (\mathrm{mod}\ #1)}

\begin{document}

\title[On $p$-adic valuations of $m$ colored $p$-ary partition functions]{On $p$-adic valuations of certain $m$ colored $p$-ary partition functions}
\author{Maciej Ulas and B{\l}a\.{z}ej \.{Z}mija}

\thanks{During the preparation of the work, the second author was a scholarship holder of the Kartezjusz program funded by the Polish National Centre for Research and Development.}

\begin{abstract}
Let $k\in\N_{\geq 2}$ and for given $m\in\Z\setminus\{0\}$ consider the sequence $(S_{k,m}(n))_{n\in\N}$ defined by the power series expansion
$$
\frac{1}{(1-x)^{m}}\prod_{i=0}^{\infty}\frac{1}{(1-x^{k^{i}})^{m}}=\sum_{n=0}^{\infty}S_{k,m}(n)x^{n}.
$$
The number $S_{k,m}(n)$ for $m\in\N_{+}$ has a natural combinatorial interpretation: it counts the number of representations of $n$ as sums of powers of $k$, where the part equal to $1$ takes one among $mk$ colors and each part $>1$ takes $m(k-1)$ colors. We concentrate on the case when $k=p\in\mathbb{P}$. Our main result is the computation of the exact value of the $p$-adic valuation of $S_{p,m}(n)$. In particular, in each case the set of values of $\nu_{p}(S_{p,m}(n))$ is finite and the maximum value is bounded by $\op{max}\{\nu_{p}(m)+1,\nu_{p}(m+1)+1\}$. Our results can be seen as a generalization of earlier work of Churchhouse and recent work of Gawron, Miska and Ulas, and the present authors.

\end{abstract}

\maketitle

\section{Introduction}\label{sec1}

Theory of partitions is a classical part of number theory. Main arithmetical questions concerning counting functions of various type of partitions are related to divisibility properties and $p$-adic behaviour for various values of $p\in\mathbb{P}$. Particular partition functions are connected with $k$-ary representations of integers. Let us recall that by a $k$-ary partition of an integer $n$ we understand the representation of $n$ into powers of $k$, i.e., the representation of the form
\begin{equation}\label{prep}
n=\sum_{i=0}^{u}k^{s_{i}},
\end{equation}
where $s_{i}\in\N$ for $i=0,\ldots, u$ and $s_{0}\leq s_{1}\leq \ldots\leq s_{u}$. By a colored $k$-ary representation we understand the $k$-ary representation, where parts take colors from finite sets.

The basic $k$-ary partition function is related to the case of $k=2$. Then, the counting function $b(n)$ of representations of $n$ in the form (\ref{prep}) is called {\it the binary partition function}. The sequence $(b(n))_{n\in\N}$ was introduced by Euler, who also noticed the identity
$$
B(x)=\prod_{n=0}^{\infty}\frac{1}{1-x^{2^{n}}}=\sum_{n=0}^{\infty}b(n)x^{n}.
$$
The series $B$ satisfies the functional equation $(1-x)B(x)=B(x^2)$. In consequence, by comparing the coefficients near $x^n$ on the both sides of the functional relation, we see that the sequence $(b(n))_{n\in\N}$ satisfies: $b(0)=b(1)=1$ and
\begin{equation*}
b(2n)=b(2n-1)+b(n),\quad b(2n+1)=b(2n).
\end{equation*}
However, it seems that the first non-trivial result concerning the sequence $(b(n))_{n\in\N}$ was obtained by Churchhouse in \cite{Chu}. In order to state his result, let $p\in\mathbb{P}$ be given and recall that the $p$-adic valuation of an integer $n$, denoted by $\nu_{p}(n)$, is just the highest power of $p$ dividing $n$, i.e.,
$$
\nu_{p}(n):=\op{max}\{k\in\N:\;p^{k}\mid n\}
$$
with the convention that $\nu_{p}(0)=+\infty$. It is clear that the $p$-adic valuation satisfies the properties: $\nu_{p}(n_{1}n_{2})=\nu_{p}(n_{1})+\nu_{p}(n_{2})$ and $\nu_{p}(n_{1}+n_{2})\geq \op{min}\{\nu_{p}(n_{1}),\nu_{p}(n_{2})\}$. Churchhouse proved that the sequence of the $2$-adic valuations of $(b(n))_{n\in\N}$ is bounded by 2. More precisely, if $n\geq 2$, we have $\nu_{2}(b(n))=2$ if and only if $n$ or $n-1$ can be written in the form $4^r(2u+1)$ for some $r\in\N_{+}$ and $u\in\N$. In the remaining cases we have $\nu_{2}(b(n))=1$. One can observe that we can shortly write
$$
\nu_{2}(b(n))=
\begin{cases}
\frac{1}{2}|t_{n}-2t_{n-1}+t_{n-2}|,& \mbox{if }\;n\geq 2\\
0,& \mbox{if }\;n\in\{0,1\}
\end{cases},
$$
where $(t_{n})_{n\in\N}$ is the Prouhet-Thue-Morse sequence (the PTM sequence for short). The PTM sequence can be defined in many equivalent ways but the most natural in the context of binary partitions is the one through generating functions. More precisely, we have
$$
T(x)=\frac{1}{B(x)}=\prod_{n=0}^{\infty}\left(1-x^{2^{n}}\right)=\sum_{n=0}^{\infty}t_{n}x^{n}.
$$
As an immediate consequence of the functional relation $T(x)=(1-x)T(x^2)$ we get the recurrence relations
$$
t_{0}=1, \quad t_{2n}=t_{n}, \quad t_{2n+1}=-t_{n}.
$$
The PTM sequence appears in many different areas of mathematics (and even physics). An interesting introduction to the properties of this sequence can be found in \cite{AllSh}. We will encounter this sequence in the sequel one more time.

In the recent paper \cite{GMU} Gawron, Miska and the first author generalized the result of Churchhouse in the following way. Let $m\in\N_{+}$ and consider the sequence $(b_{m}(n))_{n\in\N}$, where
$$
B(x)^{m}=\prod_{n=0}^{\infty}\frac{1}{\left(1-x^{2^{n}}\right)^{m}}=\sum_{n=0}^{\infty}b_{m}(n)x^{n}.
$$
The sequence $(b_{m}(n))_{n\in\N}$ has a natural combinatorial interpretation. More precisely, the number $b_m(n)$ counts the number of binary representations (\ref{prep}), where each part has one among $m$ colors. In the cited paper it was proved that the sequence  $(\nu_{2}(b_{m}(n)))_{n\in\N}$, with $m=2^{k}-1, k\in\N_{+}$, is bounded by 2 (and independent of $k$ - a rather unexpected result). This result was further generalized for $p$-ary colored partitions with $p\in\mathbb{P}_{\geq 3}$ in \cite{Ula}. Indeed, let $A_{p,m}(n)$ denotes the number of $p$-ary representations (\ref{prep}) of the number $n$, where each summand has one among $m$ colors. We proved that for each $p\in\mathbb{P}_{\geq 3}$ and $s\in\N_{+}$, the $p$-adic valuation of the number $A_{p,(p-1)(p^s-1)}(n)$ is equal to 1 for $n\geq p^s$. We also obtained some results concerning the behaviour of the sequence $(\nu_{p}(A_{p,(p-1)(up^s-1)}(n)))_{n\in\N}$ for fixed $u\in\{2,\ldots,p-1\}$ and $p\geq 3$. Let us note that the exponent $(p-1)(p^s-1)$ for $p=2$ reduces to $2^s-1$ considered in \cite{GMU} (in particular $b_{2^{s}-1}(n)=A_{2,2^s-1}(n)$ for all $n\in\N$). This shows striking difference in behaviour of the $p$-adic valuations of $A_{p,(p-1)(p^s-1)}(n)$ between the case of $p=2$ and $p\in\mathbb{P}_{\geq 3}$.

In the light of the cited results a natural question arises whether there is a $p$-ary partition type function such that its colored version has a bounded $p$-adic valuation for each fixed number of colors (not only of particular forms like in the cases considered in \cite{GMU, Ula}). The aim of this paper is to present such a function and obtain exact expression for the corresponding sequence of $p$-adic valuations.

For $k\in\N_{\geq 2}$ we consider the power series expansion
$$
G_{k}(x)=\frac{1}{1-x}\prod_{n=0}^{\infty}\frac{1}{(1-x^{k^{n}})^{k-1}}=\sum_{n=0}^{\infty}S_{k}(n)x^{n}.
$$
Then, for $m\in\Z\setminus\{0\}$ we consider the power series
$$
H_{k,m}(x)=G_{k}(x)^{m}=\frac{1}{(1-x)^{m}}\prod_{n=0}^{\infty}\frac{1}{(1-x^{k^{n}})^{m(k-1)}}=\sum_{n=0}^{\infty}S_{k,m}(n)x^{n},
$$
and the sequence $(S_{k,m}(n))_{n\in\N}$. Similarly, as in the previous cases the coefficient $S_{k,m}(n)$ for $m\in\N_{+}$, has a natural combinatorial interpretation. The number $S_{k,m}(n)$ counts the number of $k$-ary representations (\ref{prep}), such that the part equal to 1 takes one among $km$ colors and each part $>1$ takes one among $(k-1)m$ colors. In the paper we are concentrate on the case $k=p\in\mathbb{P}$. As we will see, it is possible to compute the expression for $\nu_{p}(S_{p,m}(n))$ for each $p\in\mathbb{P}$, $n\in\N$ and $m\in\Z$, explicitly. Let us note that in the case $m<0$ we are working with the sequence with varying signs and it is a non-trivial question whether the number $S_{p,m}(n)$ can be zero. Unexpectedly, we will show that this is not the case. This non-obvious observation will be consequence of our main results. We should also mention that as in previous works the case of $p=2$ and $p\in\mathbb{P}_{\geq 3}$ are quite different. Thus, we introduce the following definitions with the aim to simplify the notation a bit. More precisely, we write
$$
c_{m}(n):=S_{2,m}(n),\quad\quad d_{m}(n):=S_{p,m}(n)\quad\mbox{for fixed}\quad p\in\mathbb{P}_{\geq 3}.
$$
Let us note that we have the equality $c_{1}(n)=b(2n)$ which is a consequence of the equalities
$$
\sum_{n=0}^{\infty}b(2n)x^{n}=\frac{1}{2}(B(\sqrt{x})+B(-\sqrt{x}))=\frac{1}{1-x}\prod_{n=0}^{\infty}\frac{1}{1-x^{2^{n}}}=H_{2,1}(x).
$$
We are in position to state the content of the paper in some details.

In Section \ref{sec2} we consider the case $p=2$ and obtain the closed formula for $\nu_{2}(c_{m}(n))$ for any given $m\in\Z$. The formula is stated in Theorem \ref{mainthm} and implies that $\nu_{2}(c_{m}(n))$ is constant for even $m$ and positive $n$. In the case of odd $m$ the value of $\nu_{2}(c_{m}(n))$ is equal to 1 or $\nu_{2}(m+1)+1$, depending on whether $t_{n}\neq t_{n-1}$ or $t_{n}=t_{n-1}$, respectively.

The next two sections are devoted to the proof of the equality $\nu_{p}(d_{m}(n))=\nu_{p}(m)+1$ which holds for $n\in\N_{+}$ and each $m\in\Z\setminus\{0\}$ (Theorem \ref{parythm1} and Theorem \ref{parythm2}). According to our best knowledge our results are the first concerning the exact computation of the $p$-adic valuations of $m$ colored partitions (with non-rational generating function of course). Moreover, our results can be seen as natural generalizations of Churchhouse result. As a by-product of the employed methods we also get an unexpected fact that the sequence $\left(\frac{d_{m}(n)}{pm}\Mod{p}\right)_{n\in\N_{+}}$ is independent on $m$. Thus, a more detailed analysis of this sequence is performed in Section \ref{sec4}. The main result of this section is the proof of $p$-automaticity of the sequence $\left(\frac{d_{m}(n)}{pm}\Mod{p}\right)_{n\in\N_{+}}$ (Theorem \ref{automaticthm1}) and transcendence of the corresponding ordinary generating function (Theorem \ref{transres}).

In Section \ref{sec5} we present study of the $p$-adic behaviour of the difference $d_{m}(pn)-d_{m}(n)$. In particular we prove sufficient conditions for the inequality $\nu_{p}(d_{m}(pn)-d_{m}(n))\geq \nu_{p}(m)+3$ to hold (Theorem \ref{parythm3}).

Finally, in the last section we ask for a possible generalization of our results, present effects of some numerical computations and formulate some questions, problems and conjectures.

\section{The case of $p=2$}\label{sec2}

In order to simplify the notation a bit let us put
$$
C(x):=G_{2}(x),\quad C_{m}(x):=G_{2}(x)^{m}.
$$
Let us observe the power series $C(x)$ satisfies the functional relation $(1-x)C(x)=(1+x)C(x^2)$ . In consequence, for $m\in\Z$ we have the relation
$$
C_{m}(x)=\left(\frac{1+x}{1-x}\right)^{m}C_{m}(x^2),
$$
which will be useful in the proof of Lemma \ref{plusminus2} below. Moreover, in the sequel we will need the following functional property: for $m_{1}, m_{2}\in\Z$ we have
$$
C_{m_{1}}(x)C_{m_{2}}(x)=C_{m_{1}+m_{2}}(x).
$$

We start our investigations with a simple lemma which is a consequence of the result of Churchhouse and the product form of the series $C_{-1}(x)$.

\begin{lem}\label{plusminus1}
For $n\in\N_{+}$, we have the following equalities:
\begin{align*}
\nu_{2}(c_{1}(n)) =&\frac{1}{2}|t_{n}+3t_{n-1}|,\\
\nu_{2}(c_{-1}(n))=&\begin{cases}\begin{array}{lll}
                                   1,       &  & \mbox{if }\;t_{n}\neq t_{n-1} \\
                                   +\infty, &  & \mbox{if }\;t_{n}=t_{n-1}
                                 \end{array}
\end{cases}.
\end{align*}
\end{lem}
\begin{proof}
The first equality is an immediate consequence of the equalities $c_{
1}(n)=b(2n), \nu_{2}(b(n))=\frac{1}{2}|t_{n}-2t_{n-1}+t_{n-2}|$ and the recurrence relations satisfied by the PTM sequence $(t_{n})_{n\in\N}$, i.e., $t_{2n}=t_{n}, t_{2n+1}=-t_{n}$.

The second equality comes from the expansion
$$
C_{-1}(x)=(1-x)\prod_{n=0}^{\infty}(1-x^{2^{n}})=(1-x)\sum_{n=0}^{\infty}t_{n}x^{n}=1+\sum_{n=1}^{\infty}(t_{n}-t_{n-1})x^{n}.
$$
\end{proof}

In order to compute the 2-adic valuations of the sequence $(c_{\pm 2}(n))_{n\in\N}$ we need the following simple

\begin{lem}\label{plusminus2}
The sequences $(c_{\pm 2}(n))_{n\in\N}$ satisfy the following recurrence relations: $c_{\pm 2}(0)=1, c_{\pm 2}(1)=\pm 4$ and for $n\geq 1$ we have
\begin{align*}
c_{\pm 2}(2n)  =&\pm 2c_{\pm 2}(2n-1)-c_{\pm 2}(2n-2)+c_{\pm 2}(n)+c_{\pm 2}(n-1),\\
c_{\pm 2}(2n+1)=&\pm 2c_{\pm 2}(2n)-c_{\pm 2}(2n-1)\pm 2c_{\pm 2}(n).
\end{align*}
\end{lem}
\begin{proof}
The recurrence relations for the sequence $(c_{\pm 2}(n))_{n\in\N}$ are immediate consequence of the functional equation $C_{\pm 2}(x)=\left(\frac{1+x}{1-x}\right)^{\pm 2}C_{\pm 2}(x^2)$, which can be rewritten in an equivalent form $(1\mp x)^{2}C_{\pm 2}(x)=(1\pm x)^{2}C_{\pm 2}(x^2)$. Comparing now the coefficients on both sides of this relation we get the result.
\end{proof}

As a consequence of the recurrence relations for $(c_{\pm 2}(n))_{n\in\N}$ we get

\begin{cor}\label{plusminus2val}
For $n\in\N_{+}$ we have $c_{\pm 2}(n)\equiv 4\Mod{8}$. In consequence, for $n\in\N_{+}$ we have $\nu_{2}(c_{\pm 2}(n))=2$.
\end{cor}
\begin{proof}
The proof relies on a simple induction. Indeed, we have $c_{\pm 2}(1)=\pm 4, c_{-2}(2)=4, c_{2}(2)=12$ and thus our statement folds for $n=1, 2$. Assuming it holds for all integers $\leq n$ and applying the recurrence relations given in Lemma \ref{plusminus2} we get the result.

The second part is an immediate consequence of the obtained congruence.
\end{proof}

\begin{thm}\label{mainthm}
Let $m\in\Z\setminus\{0, -1\}$ and consider the sequence ${\bf c}_{m}=(c_{m}(n))_{n\in\N}$. Then $c_{m}(0)=1$ and for $n\in\N_{+}$ we have
\begin{equation}
\nu_{2}(c_{m}(n))=\begin{cases}\begin{array}{lll}
                                 \nu_{2}(m)+1,              &  & \mbox{if }\;m\equiv 0\Mod{2} \\
                                 1,                         &  & \mbox{if }\;m\equiv 1\Mod{2}\;\mbox{and}\;t_{n}\neq t_{n-1}\\
                                 \nu_{2}(m+1)+1,            &  & \mbox{if }\;m\equiv 1\Mod{2}\;\mbox{and}\;t_{n}=t_{n-1}
                               \end{array}
\end{cases}.
\end{equation}
\end{thm}
\begin{proof}
First of all, let us note that our theorem is true for $m=1, \pm2$. This is a consequence of Lemma \ref{plusminus1} and Corollary \ref{plusminus2val}. Let $m\in\Z$ and $|m|>2$. Because $c_{m}(0)=1$ and $c_{m}(1)=2m$, our statement is clearly true for $n=0,1$. We can assume that $n\geq 2$.

We start with the case $m=-3$. From the functional relation $C_{-3}(x)=C_{-2}(x)C_{-1}(x)$ we immediately get the identity
$$
c_{-3}(n)=\sum_{i=0}^{n}c_{-1}(i)c_{-2}(n-i)= c_{-2}(n)+t_{n}-t_{n-1}+ \sum_{i=1}^{n-1}(t_{i}-t_{i-1})c_{-2}(n-i).
$$
Let us observe that for $i\in\{1,\ldots, n-1\}$, from Lemma \ref{plusminus1} and Corollary \ref{plusminus2val}, we obtain the inequality
$$
\nu_{2}((t_{i}-t_{i-1})c_{-2}(n-i))\geq 3.
$$
In consequence, from Lemma \ref{plusminus2}, we get
$$
c_{-3}(n)\equiv c_{-2}(n)+t_{n}-t_{n-1}\equiv 4+t_{n}-t_{n-1}\Mod{8}.
$$
It is clear that $4+t_{n}-t_{n-1}\not\equiv 0\Mod{8}$. Thus, we get the equality $\nu_{2}(c_{-3}(n))=\nu_{2}(4+t_{n}-t_{n-1})$ and the result follows for $m=-3$.

We are ready to prove the general result. At first, we prove it for even numbers $m$. We proceed by induction on $|m|$. % (which depends on the remainder of $m\Mod{4}$).
As we have already proved, our theorem is true for $m=\pm 2$. %$m=1, \pm 2$ and $m=-3$.
Let us assume that it is true for each $m$ satisfying $|m|<2M$. Let $|m|= 2M$ and write $m=4k+r$ for some $k\in\Z$ and $r\in\{0,2\}$.

If $m=4k$, then from the identity $C_{4k}(x)=C_{2k}(x)^2$ we get the expression
$$
c_{4k}(n)=2c_{2k}(n)+\sum_{i=1}^{n-1}c_{2k}(i)c_{2k}(n-i).
$$
From the induction hypothesis (notice, that $|4k|>|2k|$) we have $\nu_{2}(c_{2k}(i)c_{2k}(n-i))=2(\nu_{2}(2k)+1)>\nu_{2}(2c_{2k}(n))=\nu_{2}(2k)+2$. In consequence
$\nu_{2}(c_{m}(n))=\nu_{2}(c_{4k}(n))=\nu_{2}(2c_{2k}(n))=\nu_{2}(2k)+2=\nu_{2}(4k)+1$. The obtained equality finishes the proof in the case $m\equiv 0\Mod{4}$.

Similarly, if $m=4k+2$ is positive, we use the identity $C_{4k+2}(x)=C_{4k}(x)C_{2}(x)$, and get
$$
c_{4k+2}(n)=c_{2}(n)+c_{4k}(n)+\sum_{i=1}^{n-1}c_{4k}(i)c_{2}(n-i).
$$
Observe, that $|4k+2|>|4k|$. From the equalities $\nu_{2}(c_{2}(n))=\nu_{2}(2)+1$ and $\nu_{2}(c_{4k}(n))=\nu_{2}(4k)+1, n\in\N_{+}$, we get  $\nu_{2}(c_{4k}(i)c_{2}(n-i))=\nu_{2}(k)+5$ for each $i\in\{1,\ldots,n-1\}$. Thus $\nu_{2}(c_{2}(n)+c_{4k}(n))=\nu_{2}(c_{2}(n))=2=\nu_{2}(4k+2)+1$.

If $m=4k+2$ is negative, we use the identity $C_{4k+2}(x)=C_{4(k+1)}(x)C_{-2}(x)$ and proceed in exactly the same way (here we can use the induction hypothesis because $|4k+2|>|4(k+1)|$ since $k<0$).

Thus we can conclude, that the statement is true for all even numbers $m$. We go to odd ones. We have already showed the theorem for $m=-1$ and $m=-3$. Similarly as before, let us write $m=4k+r$ for some $k\in\Z$ and $r\in\{1,3\}$.

If $m=4k+1>0$, then we use the identity $C_{4k+1}(x)=C_{4k}(x)C_{1}(x)$ and get
$$
c_{4k+1}(n)=c_{4k}(n)+c_{1}(n)+\sum_{i=1}^{n-1}c_{4k}(i)c_{1}(n-i).
$$
From the even case we have $\nu_{2}(c_{4k}(i)c_{1}(n-i))\geq \nu_{2}(4k)+2\geq 4$. Moreover, for $n\in\N_{+}$ we have $\nu_{2}(c_{1}(n))\in\{1,2\}$. Thus
\begin{equation*}
\nu_{2}(c_{4k}(n)+c_{1}(n))=\nu_{2}(c_{1}(n))=\begin{cases}\begin{array}{lll}
                                 1,            &  & \mbox{if }\;t_{n}\neq t_{n-1}\\
                                 2,            &  & \mbox{if }\;t_{n}=t_{n-1}
                               \end{array}
                               \end{cases}.
\end{equation*}
as we claimed.

If $m=4k+1<0$, we write $m=4(k+1)-3$ and use the identity $C_{4k+1}(x)=C_{4(k+1)}(x)C_{-3}(x)$. Next, using the obtained expression for $\nu_{2}(c_{-3}(n))$ and $\nu_{2}(c_{4(k+1)}(n))$ and the same reasoning as in the positive case we get the result.

If $m=4k+3>0$ we use the identity $C_{4k+3}(x)=C_{4(k+1)}(x)C_{-1}(x)$ which leads us to the expression
$$
c_{4k+3}(n)=c_{4(k+1)}(n)+c_{-1}(n)+\sum_{i=1}^{n-1}c_{4(k+1)}(i)c_{-1}(n-i).
$$
It is clear that $\nu_{2}(c_{4(k+1)}(i)c_{-1}(n-i))>\nu_{2}(c_{4(k+1)}(n)+c_{-1}(n))$ for each $n\in\N_{+}$ and $i\in\{1,\ldots,n-1\}$. In consequence, by induction hypothesis
\begin{align*}
\nu_{2}(c_{4k+3}(n))&=\nu_{2}(c_{4(k+1)}(n)+c_{-1}(n))\\
                    &=\begin{cases}\begin{array}{lll}
                                 1,                                 &  & \mbox{if }\;t_{n}\neq t_{n-1}\\
                                 \nu_{2}(c_{4(k+1)}(n)),            &  & \mbox{if }\;t_{n}=t_{n-1}
                               \end{array}
                               \end{cases}\\
                    &=\begin{cases}\begin{array}{lll}
                                 1,                            &  & \mbox{if }\;t_{n}\neq t_{n-1}\\
                                 \nu_{2}(4k+3+1)+1,            &  & \mbox{if }\;t_{n}=t_{n-1}
                               \end{array}
                               \end{cases}.
\end{align*}

Finally,, if $m=4k+3<0$, then we write $4k+3=4(k+1)-1$ and employ the identity $C_{4k+3}(x)=C_{4(k+1)}(x)C_{-1}(x)$. Using the same reasoning as in the previous cases we get the result.
\end{proof}

The explicit form of the value of $\nu_{2}(c_{m}(n))$ for $m$ odd allow us to prove the following.
\begin{cor}\label{cor0}
Let $m\in\Z$ be odd and $n\in\N_{+}$. Then $\nu_{2}(c_{m}(2n)-c_{m}(n))=1$.
\end{cor}
\begin{proof}
Let us observe that the necessary condition for the inequality $\nu_{2}(c_{m}(2n)-c_{m}(n))>1$ to hold is the condition
$$
t_{n}\neq t_{n-1}\Longrightarrow t_{2n}\neq t_{2n-1}\quad\mbox{or}\quad t_{n}=t_{n-1}\Longrightarrow t_{2n}=t_{2n-1}.
$$
However, if $t_{n}\neq t_{n-1}$, then $t_{n}=-t_{n-1}$ and thus $t_{2n}-t_{2n-1}=t_{n}+t_{n-1}=0$. Similarly, the equality $t_{n}=t_{n-1}$ implies that $t_{n}+t_{n-1}=2t_{n}$ and thus $t_{2n}-t_{2n-1}=2t_{n}\neq 0$. In other words, if $\nu_{2}(c_{m}(n))=1$ then $\nu_{2}(c_{m}(2n))>1$ and vice versa. In consequence we have the equality $\nu_{2}(c_{m}(2n)-c_{m}(n))=1$ and get the result.
\end{proof}

\begin{cor}
Let $m\in\N_{+}$ and suppose that $N\in\N_{+}$ satisfy the condition $\nu_{2}(N)>\nu_{2}(m+1)+1$. Then, the congruence $c_{m}(n)\equiv 0\pmod{N}$ has no solutions.
\end{cor}

Let us observe that the sequence $(\nu_{2}(c_{m}(n)))_{n\in\N}$ (for $m$ odd) can be seen alternatively as a solution of the recurrence relation: $w_{m}(0)=0, w_{m}(1)=1, w_{m}(2)=\nu_{2}(m+1)+1$ and
$$
w_{m}(2n+1)=1,\quad w_{m}(4n)=w_{m}(n),\quad w_{m}(4n+2)=\nu_{2}(m+1)+1.
$$
The above recurrence is an immediate consequence of our explicit formula given in Theorem \ref{mainthm}. We use it in order to show that the sequence $(\nu_{2}(c_{m}(n)))_{n\in\N}$ is $2$-automatic. However, before we do this we need to recall the notion of a {\it $k$-automatic sequence}. More precisely, let $k\in\N_{\geq 2}$ be given. We say that the sequence ${\bf \eps}=(\eps_{n})_{n\in\N}$ is {\it $k$-automatic} if and only if the following set
$$
K_{k}({\bf \eps})=\{(\eps_{k^{i}n+j})_{n\in\N}:\;i\in\N\;\mbox{and}\;0\leq j<k^{i}\},
$$
called the $k$-kernel of ${\bf \eps}$, is finite. If the $k$-kernel of the sequence ${\bf \eps}$ is infinite but finitely generated then we say that our sequence is {\it $k$-regular}.

Let us also recall that one of equivalent conditions for automaticity, is the following result of Christol: {\it let $p\in\mathbb{P}$ and $q$ be a power of $p$, and ${\bf a}=(a_{n})_{n\in\N}$ be a sequence over $\mathbb{F}_{q}$. Then ${\bf a}$ is $p$-automatic if and only if the formal power series $\sum_{n=0}^{\infty}a_{n}x^{n}$ is algebraic over $\mathbb{F}_{q}(x)$}. The proof can be found in \cite{Chr} or in recent monograph \cite[Theorem 12.2.5]{AllSh}.

We are ready to prove that the sequence $(\nu_{2}(c_{m}(n)))_{n\in\N}$ is $2$-automatic for each $m\in\Z\setminus\{0,-1\}$.

\begin{thm}
For each $m\in\Z\setminus\{0,-1\}$ the sequence $\left(\nu_{2}(c_{m}(n))\right)_{n\in\N_{+}}$ is $2$-automatic.
\end{thm}
\begin{proof}
The result is obvious if $m$ is even, because in this case the sequence is ultimately constant by Theorem \ref{mainthm}. Let us assume that $m$ is odd and denote $\tilde{c}_{m}(n):=\nu_{2}(c_{m}(n))$ for $n\in\N$, and $\tilde{C}_{m}(x):=\sum_{n=0}^{\infty}\tilde{c}_{m}(n)x^{n}$. Theorem \ref{mainthm} and the remark given above imply the relations:
\begin{align*}
\left\{\begin{array}{l}
\tilde{c}_{m}(2n+1)=  1, \\
\tilde{c}_{m}(4n+2)=  \nu_{2}(m+1)+1, \\
\tilde{c}_{m}(4n)= \tilde{c}_{m}(n).
\end{array}\right.
\end{align*}
Hence,
\begin{align*}
\tilde{C}_{m}(x)= & \sum_{n=0}^{\infty}\tilde{c}_{m}(4n)x^{4n}+\sum_{n=0}^{\infty}\tilde{c}_{m}(4n+2)x^{4n+2}+\sum_{n=0}^{\infty}\tilde{c}_{m}(2n+1)x^{2n+1} \\
= & \sum_{n=0}^{\infty}\tilde{c}_{m}(n)x^{4n}+\left(\nu_{2}(m+1)+1\right)\sum_{n=0}^{\infty}x^{4n+2}+\sum_{n=0}^{\infty}x^{2n+1} \\
= & \tilde{C}_{m}(x^{4})+\left(\nu_{2}(m+1)+1\right)\frac{x^{2}}{1-x^{4}}+\frac{x}{1-x^{2}}.
\end{align*}
Equivalently:
\begin{align*}
\tilde{C}_{m}(x^{4})-\tilde{C}_{m}(x)+\frac{x^{3}+\left(\nu_{2}(m+1)+1\right)x^2+1}{1-x^{4}}=0,
\end{align*}
so $\tilde{C}_{m}(x)$ is algebraic over $\F_{2}(x)$ and the result follows from Christol's Theorem.
\end{proof}

It is worth to note, that a sequence is $k$-automatic if and only if there exists an automaton for it. Roughly speaking, automaton is a procedure that allow us to compute for each $n\in\N$ the $n$-th member of a sequence by using only the digits in the (unique) representation of $n$ in the base $k$. For more precise definition see \cite{AllSh1}. We can either use the convention to read the digits of the base-$k$ representation of $n$ from the highest power of $k$ that appear in the representation to the lowest one (that is, from the right to the left), or conversely, that is, from the lowest (i.e., from the $0$'th power) to the highest one (that is, from the left to the right). It is known, that a given sequence is $k$-automatic in the first convention if and only if it is $k$-automatic in the second one, but automatons may be different depending on the convention.

We can present automatons generating the sequence $(\nu_{p}(c_{m}(n)))_{n\in\N_{+}}$ in both conventions. In the case of even $m$ it is very simple. For odd $m$, they are presented below.

\begin{center}
	\begin{tikzpicture}[->, shorten >= 1pt, node distance=2.5 cm, on grid, auto]
	\node[state, initial] (c_0) {$-$};
	\node[state, inner sep=1pt] (c_1) [right=5cm of c_0] {$1$};
	\node[state, inner sep=1pt] (c_2) [below= 2.5cm of c_0] {$-$};
	\node[state, inner sep=1pt] (c_3) [below right=2.5cm and 5cm of c_0] {$\nu_{2}(m+1)+1$};
	\path[->]
	(c_0) edge [left, above] node {$1$} (c_1)
	edge [bend left] node {$0$} (c_2)
	(c_1) edge [loop right] node {$0,1$} (c_1)
	(c_2) edge [bend left] node {$0$} (c_0)
	edge [left, above] node {$1$} (c_3)
	(c_3) edge [loop right] node {$0,1$} (c_3);	
	\end{tikzpicture}
\end{center}
\begin{center}
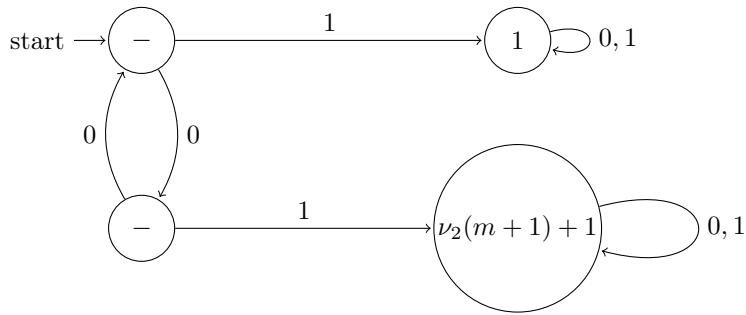
	Figure 1. Automaton generating $(\nu_{2}(c_{m}(n)))_{n\in\N_{+}}$ for odd $m$ when reading digits of $n$ from the left to the right.
\end{center}

\bigskip

\begin{center}
	\begin{tikzpicture}[->, shorten >= 1pt, node distance=2.5 cm, on grid, auto]
	\node[state, initial] (c_0) {$1$};
	\node[state, inner sep=1pt] (c_1) [right=5cm of c_0] {$\nu_{2}(m+1)+1$};
	\path[->]
	(c_0) edge [loop above] node {$1$} (c_0)
	edge [bend left] node {$0$} (c_1)
	(c_1) edge [bend left] node {$0,1$} (c_0);
	\end{tikzpicture}
\end{center}
\begin{center}
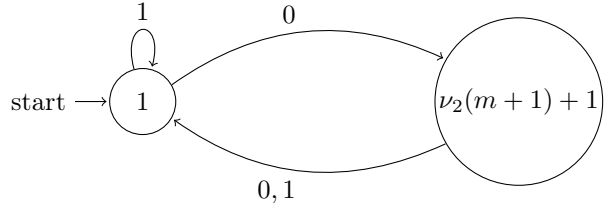
	Figure 2. Automaton generating $(\nu_{2}(c_{m}(n)))_{n\in\N_{+}}$ for odd $m$ when reading digits of $n$ from the right to the left.
\end{center}

\bigskip

\begin{rem}
	{\rm The second automaton is very similar to the one generating so called {\it period-doubling sequence} given by $e_{n}=\nu_{2}(n+1)\Mod{2}$ for $n\in\N$. This sequence was considered for example in \cite{Dam}. This similarity is not a coincidence, because the period-doubling sequence can be defined also as the sequence $e_{n}=1-(s_{2}(n)-s_{2}(n-1))\Mod{2}$, where $s_{2}(n)$ is the number of 1's in the unique binary expansion of $n$. On the other hand, the value of the quantity $\nu_{2}(c_{m}(n))$ depends on the difference $t_{n}-t_{n-1}$ and hence, because of the equality $t_{n}=(-1)^{s_{2}(n)}$, on the difference $s_{n}-s_{n-1}$.
	
		From the above discussion one can observe, that $e_{n}=1$ if and only if $t_{n}\neq t_{n-1}$. Thus, Theorem \ref{mainthm} for odd $m$ can be rewritten as
	\begin{align*}
		\nu_{2}(c_{m}(n))= & \nu_{2}(n+1)\Mod{2}+(1-\nu_{2}(n+1)\Mod{2})(\nu_{2}(m+1)+1) \\
		= & 1+(1-\nu_{2}(n+1)\Mod{2})\nu_{2}(m+1)
	\end{align*}
	for all $n\geq 1$.
	}
\end{rem}

\begin{cor}\label{cor1}
If $m<-1$, then $c_{m}(n)\neq 0$ for $n\in\N$.
\end{cor}
\begin{proof}
From our description of the 2-adic valuation of $c_{m}(n)$ given in Theorem \ref{mainthm}, we see that $\nu_{2}(c_{m}(n))\neq +\infty$ for $n\in\N$. This is equivalent with the non-vanishing of $c_{m}(n)$ for $n\in\N$ and hence the result.
\end{proof}

Let us recall that in \cite{GMU} the sequence $(t_{m}(n))_{n\in\N}$, $m\in\N_{+}$, was investigated. Here $t_{m}(n)$ is the Cauchy convolution of $m$-copies of the PTM sequence, i.e.,
$$
t_{m}(n)=\sum_{i_{1}+\ldots+i_{m}=n}t_{i_{1}}\ldots t_{i_{m}}.
$$
We characterized the solutions of the equation $t_{3}(n)=0$ in terms of the expansion of $n$ in base 4. In particular, the equation $t_{3}(n)=0$ has {\it infinitely} many solutions. Moreover, $t_{m}(n)\neq 0$ for $m$ being power of 2. We also conjectured the non-vanishing of $t_{m}(n)$ for $m\neq 3$ but we were unable to prove such a statement. Thus, the above corollary also shows strong difference of the sequences $(t_{m}(n))_{n\in\N}$ and $(c_{-m}(n))_{n\in\N}$ for $m\in\N_{+}$.

The result obtained in Theorem \ref{mainthm} can be also used to prove the following

\begin{cor}\label{cor2}
Let $m\in\N_{+}$ and $U_{m}(n)$ denotes the number of binary partitions of $n$ such the part equal to 1 takes one among $2m+1$ colors and parts $>1$ take one among $m$ colors. Then $U_{m}(n)\equiv 1\Mod{2}$ for $n\in\N$.
\end{cor}
\begin{proof}
From the definition of $U_{m}(n)$ we immediately deduce that $U_{m}(n)=\sum_{i=0}^{n}c_{m}(i)$. Because $c_{m}(0)=1$ and from Theorem \ref{mainthm} we know that $c_{m}(n)\equiv 0\Mod{2}$ for $n\in\N_{+}$, we get that $C_{m}(n)$ is odd for $n\in\N$.
\end{proof}

\section{The case of $p\in\mathbb{P}_{\geq 3}$}\label{sec3}

Let $p\in\mathbb{P}_{\geq 3}$ be fixed. In this section we are interested in the computation of the exact formula for the $p$-adic valuations of the elements of the sequence $(d_{m}(n))_{n\in\N}$, where
$$
H_{m,p}(x)=G_{p}(x)^{m}=\frac{1}{(1-x)^{m}}\prod_{n=0}^{\infty}\frac{1}{\left(1-x^{p^{n}}\right)^{(p-1)m}}=\sum_{n=0}^{\infty}d_{m}(n)x^{n}.
$$
and $m\in\Z\setminus\{0\}$ is fixed. Because $p$ is fixed in the sequel in order to simplify the notation a bit we introduce the quantity:
$$
H_{m}(x):=H_{p,m}(x).
$$
Let us recall that the number $d_{m}(n)$ counts the number of $p$-ary representations of $n$, where each part $>1$ can take $(p-1)m$ colors and additionally the part equal to 1 can take $pm$ colors. Moreover, in the sequel we will frequently use the following notation:
$$
[A]=\begin{cases}\begin{array}{lll}
                   1, &  & \mbox{if}\; A\;\mbox{is satisfied}  \\
                   0, &  & \mbox{otherwise}
                 \end{array}
\end{cases},
$$
where $A$ is given logical formula.

We start with the proof of recurrence relations satisfied by the sequence $(d_{m}(n))_{n\in\N}$ for $m\in\N_{+}$.

\begin{lem}\label{recrel}
Let $m\in\N_{+}$. Then $d_{m}(0)=1$ and for $n\geq 1$ we have the following recurrence relation
\begin{equation*}%\label{paryrecur}
d_{m}(n)=-\sum_{i=1}^{\op{min}\{pm,n\}}(-1)^{i}\binom{pm}{i}d_{m}(n-i)+[p\mid n]\left(\sum_{i=0}^{\op{min}\{m,\frac{n}{p}\}}(-1)^{i}\binom{m}{i}d_{m}\left(\frac{n}{p}-i\right)\right).
\end{equation*}
\end{lem}
\begin{proof}
In order to get the result it is enough to note the functional identity
$$
(1-x)^{pm}H_{m}(x)=(1-x^{p})^{m}H_{m}(x^p).
$$
Indeed, we have the following equalities
\begin{align*}
(1-x)^{pm}H_{m}(x)&=\left(\sum_{i=0}^{pm}(-1)^{i}\binom{pm}{i}x^{i}\right)\left(\sum_{n=0}^{\infty}d_{m}(n)x^{n}\right)=\sum_{n=0}^{\infty}\left(\sum_{i=0}^{\op{min}\{pm,n\}}(-1)^{i}\binom{pm}{i}d_{m}(n-i)\right)x^{n},\\
(1-x^{p})^{m}H_{m}(x^p)&=\left(\sum_{i=0}^{m}(-1)^{i}\binom{m}{i}x^{pi}\right)\left(\sum_{n=0}^{\infty}d_{m}(n)x^{pn}\right)=\sum_{n=0}^{\infty}\left(\sum_{i=0}^{\op{min}\{m,n\}}(-1)^{i}\binom{m}{i}d_{m}(n-i)\right)x^{pn}.    \\
\end{align*}
Comparing now the coefficients near $x^{n}$ on both sides we see that if $p\nmid n$, then
$$
\sum_{i=0}^{\op{min}\{pm,n\}}(-1)^{i}\binom{pm}{i}d_{m}(n-i)=0
$$
and if $p\mid n$, then we have
$$
\sum_{i=0}^{\op{min}\{pm,n\}}(-1)^{i}\binom{pm}{i}d_{m}(n-i)=\sum_{i=0}^{\op{min}\{m,\frac{n}{p}\}}(-1)^{i}\binom{m}{i}d_{m}\left(\frac{n}{p}-i\right).
$$
Solving for $d_{m}(n)$ we get the result.
\end{proof}

In the sequel we will heavily use the notion of congruences between power series. More precisely, let $f(x)=\sum_{n=0}^{\infty}a_{n}x^{n}$ and $g(x)=\sum_{n=0}^{\infty}b_{n}x^{n}$ are formal power series with coefficients in $\Z$ and $M\in\N_{\geq 2}$ be given. We say that $f, g$ are {\it congruent modulo} $M$ if and only if for all $n$ the coefficients of $x^{n}$ in both series are congruent modulo $M$. In other words
$$
f\equiv g\Mod{M}\Longleftrightarrow\forall n\in\N:\;a_{n}\equiv b_{n}\Mod{M}.
$$

It is an easy exercise to prove that for any given $f, F, g, G\in\Z[[x]]$ satisfying
$f\equiv g\Mod{M}$ and $F\equiv G\Mod{M}$, we have $f\pm F\equiv g\pm G\Mod{M}$ and $fF\equiv gG\Mod{M}$. Moreover, if $f(0), g(0)\in\{-1,1\}$, then the series for $1/f$ and  $1/g$ have integer coefficients and then $\frac{1}{f}\equiv \frac{1}{g}\Mod{M}$.
Consequently, in this case we have $f^{k}\equiv g^{k}\Mod{M}$ for all $k\in\Z$.

Now we are ready to state the following.

\begin{lem}\label{parylem1}
Let $m=p^{\alpha}k$. Then
\begin{align*}
d_{m}(n)\equiv 0\Mod{p^{\alpha +1}}
\end{align*}
for all $n\geq 1$.
\end{lem}
\begin{proof}
From the relation $(1-x)^{p}\equiv 1-x^{p}\Mod p$ we get $(1-x)^{p^{\alpha+1}}\equiv (1-x^{p})^{p^{\alpha}}\Mod{p^{\alpha+1}}$, and hence,
$$
(1-x)^{pm}\equiv (1-x^{p})^{m}\Mod{p^{\alpha+1}}.
$$
This, together with the equality $(1-x)^{pm}H_{m}(x)=(1-x^{p})^{m}H_{m}(x^{p})$ gives
$$
H_{m}(x)\equiv H_{m}(x^{p})\Mod{p^{\alpha+1}}.
$$
The above relation simply implies $d_{m}(n)\equiv 0\Mod{p^{\alpha +1}}$ if $p\nmid n$, and $d_{m}(n)\equiv d_{m}(n/p)\Mod{p^{\alpha +1}}$ if $p\mid n$. Hence, if $n=p^{\beta}l$, then
\begin{align*}
d_{m}(n)=d_{m}(p(p^{\beta -1}l))\equiv d_{m}(p^{\beta-1}l)\equiv\ldots\equiv d_{m}(pl)\equiv d_{m}(l)\equiv 0\Mod{p^{\alpha +1}}.
\end{align*}
Therefore, the only coefficient of the series $H_{m}(x)$ not divisible by $p^{\alpha +1}$ is $d_{m}(0)=1$.
\end{proof}

\begin{lem}\label{parylem2}
Let $m=p^{\alpha}k$. Then
\begin{align*}
d_{m}(pn)\equiv d_{m}(n)\Mod{p^{\alpha +2}}
\end{align*}
for all $n\geq 1$.
\end{lem}
\begin{proof}
We use the induction argument. For $n=1$ we have by Lemma \ref{recrel}:
\begin{align*}
d_{m}(p)= & -\sum_{i=1}^{p}(-1)^{i}\binom{pm}{i}d_{m}(p-i)+\sum_{i=0}^{1}(-1)^{i}\binom{m}{i}d_{m}(1-i) \\
 = & -\sum_{i=1}^{p}(-1)^{i}\binom{pm}{i}d_{m}(p-i)+d_{m}(1)-md_{m}(0).
\end{align*}
Observe, that if $p\nmid i$ then $p\mid \binom{pm}{i}$. From Lemma \ref{parylem1} we also have $p^{\alpha+1}\mid d_{m}(p-i)$ for $i=1,\ldots ,p-1$. Hence, $\binom{pm}{i}d_{m}(p-i)\equiv 0\Mod{p^{\alpha+2}}$ for $i=1,\ldots ,p-1$, so
\begin{align*}
d_{m}(p)\equiv \binom{pm}{p}d_{m}(0)+d_{m}(1)-md_{m}(0)\equiv md_{m}(0)+d_{m}(1)-md_{m}(0)\equiv d_{m}(1)\Mod{p^{\alpha +2}}.
\end{align*}
	
Let us assume that the statement is true for all numbers less than some number $n$. We want to prove it for $n$. Lemma \ref{recrel} implies
\begin{align*}
d_{m}(pn)= & -\sum_{i=1}^{\min\{pm,pn\}}(-1)^{i}\binom{pm}{i}d_{m}(pn-i)+\sum_{i=0}^{\min\{m,n\}}(-1)^{i}\binom{m}{i}d_{m}(n-i) \\
 \equiv & -\sum_{i=1}^{\min\{m,n\}}(-1)^{pi}\binom{pm}{pi}d_{m}(pn-pi)+\sum_{i=0}^{\min\{m,n\}}(-1)^{i}\binom{m}{i}d_{m}(n-i) \\
 = & -\sum_{i=1}^{\min\{m,n\}}(-1)^{i}\binom{pm}{pi}d_{m}(p(n-i))+\sum_{i=0}^{\min\{m,n\}}(-1)^{i}\binom{m}{i}d_{m}(n-i) \\
 \equiv & -\sum_{i=1}^{\min\{m,n\}}(-1)^{i}\binom{m}{i}d_{m}(n-i)+\sum_{i=0}^{\min\{m,n\}}(-1)^{i}\binom{m}{i}d_{m}(n-i) \\
 = & d_{m}(n)\Mod{p^{\alpha +2}}.
\end{align*}
In the above computations we used the induction hypothesis, Lemma \ref{parylem1} and the congruence $\binom{pm}{pi}\equiv \binom{m}{i}\Mod p$ which is a consequence of Lucas theorem. The result follows.
\end{proof}

%It seems that  for $m=1$ and $p\geq 5$ we have stronger property (which is of independent interest)
%d_{p}(p^{k+1}n)\equiv d_{p}(p^{k}n)\Mod{p^{k+3}}.
%$$

\begin{lem}\label{parylem3}
Let $m=p^{\alpha}$. Then
\begin{align*}
d_{m}(n)\equiv \left\{\begin{array}{ll}
(-1)^{n+1}\binom{pm}{n}\Mod{p^{\alpha+2}}& \text{if } n\leq pm, \\
d_{m}(n-pm)\Mod{p^{\alpha+2}}, & \text{if } n>pm
\end{array}\right.
\end{align*}
for all $n\geq 1$ such that $p\nmid n$.
\end{lem}
\begin{proof}
Let $n\in\N_{+}$ and assume that $p\nmid n$. Then Lemma \ref{recrel} implies the equality
\begin{align*}
d_{m}(n)=\sum_{i=1}^{\min\{pm,n\}}(-1)^{i+1}\binom{pm}{i}d_{m}(n-i).
\end{align*}
If $p\nmid i$ then $p|\binom{pm}{i}$, and by Lemma \ref{parylem1} we get the divisibility property $p^{\alpha +1}|d_{m,p}(n-i)$ which holds for $n>i$ . Therefore, if $n<pm$, then the only term in the above sum that may not be divisible by $p^{\alpha+2}$ is the one corresponding to $i=n$, i.e., $(-1)^{n+1}\binom{pm}{n}d_{m,p}(0)=(-1)^{n+1}\binom{pm}{n}$. Similarly, if $n>pm$, then the only such term can be the one corresponding to $i=pm$, i.e., $(-1)^{pm+1}\binom{pm}{pm}d_{m}(n-pm)=d_{m}(n-pm)$.
\end{proof}

Now we are ready to prove the following general result concerning the expression of $\nu_{p}(d_{m}(n))$ for $m\in\N_{+}$.
\begin{thm}\label{parythm1}
Let $p\in\mathbb{P}_{\geq 3}$ and $m\in\N$. Let $n\in\N_{\geq 1}$ and $n=n_{s}p^{s}+\ldots +n_{t}p^{t}$ be (the unique) representation of $n$ in base $p$. Here $s\leq t$ and $n_{s}\neq 0$. Then
	\begin{align*}
	d_{m}(n)\equiv pm\left(n_{s}^{-1}\Mod p\right)\Mod{p^{\nu_{p}(m)+2}}.
	\end{align*}
	In particular,
	\begin{equation*}
	\nu_{p}(d_{m}(n))=\begin{cases}\begin{array}{lll}
	0,           &  & \mbox{if }\;n=0, \\
	\nu_{p}(m)+1,&  & \mbox{if }\; n\geq 1.
	\end{array}
	\end{cases}
	\end{equation*}
\end{thm}
\begin{proof}
At the beginning observe, that for each $m$ we have $d_{m}(0)=1$, so $\nu_{p}(d_{m}(0))=0$.

Let us assume that $n\geq 1$ and $m=p^{\alpha}$ for some $\alpha\geq 0$. By Lemma \ref{parylem1} it is reasonable to consider the sequence ${\bf a}_{\alpha}=(a_{\alpha}(n))_{n\in\N_{+}}$ of elements of $\Z/p\Z$, where
$$
a_{\alpha}(n):=\frac{d_{m}(n)}{p^{\alpha +1}}\Mod{p}.
$$
Lemmas \ref{parylem2} and \ref{parylem3} imply the following recurrence relations:
\begin{align*}
\left\{\begin{array}{ll}
a_{\alpha}(n)=\frac{(-1)^{n+1}}{p^{\alpha +1}}\binom{p^{\alpha +1}}{n}\Mod{p}, & \text{if } n<p^{\alpha +1} \text{ and } p\nmid n, \\
a_{\alpha}(n)=a_{\alpha}(n-p^{\alpha+1}), & \text{if } n>p^{\alpha +1} \text{ and } p\nmid n, \\
a_{\alpha}(pn)=a_{\alpha}(n).
\end{array}\right.
\end{align*}
Observe, that if $n<p^{\alpha +1}$ and $p\nmid n$ then
\begin{align*}
\frac{(-1)^{n+1}}{p^{\alpha +1}}\binom{p^{\alpha +1}}{n}= & (-1)^{n+1}\frac{(p^{\alpha+1}-1)(p^{\alpha+1}-2)\cdot\ldots\cdot (p^{\alpha+1}-n+1)}{n!} \\
 \equiv & (-1)^{n+1}\frac{(-1)(-2)\cdot\ldots\cdot (-n+1)}{n!}=(-1)^{n+1}\frac{(-1)^{n-1}(n-1)!}{n!}=\frac{1}{n}\Mod{p}.
\end{align*}
Thus $a_{\alpha}(n)=\frac{1}{n}\Mod p$ for $n<p^{\alpha +1}$ such that $p\nmid n$. If $n$ is arbitrary, then recurrence relations for $a_{\alpha}(n)$ imply
\begin{align*}
a_{\alpha}(n)=a_{\alpha}\left(\frac{n}{p^{\nu_{p}(n)}}\right)=a_{\alpha}\left(\frac{n}{p^{\nu_{p}(n)}}\Mod{p^{\alpha+1}}\right)=\frac{1}{\frac{n}{p^{\nu_{p}(n)}}\Mod{p^{\alpha+1}}}\Mod{p}.
\end{align*}
Write $n=n_{s}p^{s}+\ldots +n_{t}p^{t}$, $s\leq t$, $n_{s}>0$. Then
\begin{align*}
\frac{n}{p^{\nu_{p}(n)}}\Mod{p^{\alpha+1}}=n_{s}+\ldots +n_{s+\alpha}p^{\alpha}=:N
\end{align*}
(we define $n_{i}=0$ for $i>t$ if required). Observe, that for $M=n_{s}^{-1}\Mod p$ we have $MN\equiv 1\Mod p$. Hence, $a_{\alpha}(n)=n_{s}^{-1}\Mod p$.

From the definition we get
\begin{align*}
d_{m}(n)\equiv p^{\alpha+1}a_{\alpha}(n)=pm\left(n_{s}^{-1}\Mod p\right)\Mod{p^{\alpha+2}}.
\end{align*}
Hence the result is true if $m$ is a power of $p$ or $m=1$.

If $m=p^{\alpha}k$, where $k\in\N_{+}$ and $p\nmid k$, then
\begin{align*}
H_{m}(x)=\left(H_{p^{\alpha}}(x)\right)^{k},
\end{align*}
so
\begin{align*}
d_{m}(n)=\sum_{a_{1}+\ldots +a_{k}=n}d_{p^{\alpha}}(a_{1})\cdot\ldots\cdot d_{p^{\alpha}}(a_{k}).
\end{align*}
We note that, if at least two numbers among $a_{1},\ldots ,a_{k}$ are non--zero, then by Lemma \ref{parylem1} we get $d_{p^{\alpha}}(a_{1})\cdot\ldots\cdot d_{p^{\alpha}}(a_{k})\equiv 0\Mod{p^{\alpha+2}}$. Hence, the expression $d_{p^{\alpha}}(a_{1})\cdot\ldots\cdot d_{p^{\alpha}}(a_{k})$ may not be divisible by $p^{\alpha +2}$ only if $a_{j}=n$ for some $j$ and $a_{i}=0$ for all $i\neq j$. Thus
\begin{align*}
d_{m}(n)\equiv kd_{p^{\alpha}}(n)\equiv \frac{kp^{\alpha+1}}{n_{s}}=pm\left(n_{s}^{-1}\Mod p\right) \Mod{p^{\alpha +2}},
\end{align*}
and our result follows.
\end{proof}

One can ask, whether it is possible to give a theorem similar to Theorem \ref{parythm1} but for negative integers $m$? In order to give an answer, let us introduce the sequence $(D_{p}(n))_{n=1}^{\infty}$ defined as
\begin{align*}
	\Delta_{p}(x):=\prod_{n=0}^{\infty}(1-x^{p^{n}})^{p-1}=\sum_{n=0}^{\infty}D_{p}(n)x^{n}.
\end{align*}
Observe, that
\begin{align*}
	H_{-1}(x)=(1-x)\Delta_{p}(x)=1+\sum_{n=1}^{\infty}\left(D_{p}(n)-D_{p}(n-1)\right)x^{n}.
\end{align*}
Therefore, $d_{-1}(n)=D_{p}(n)-D_{p}(n-1)$ for $n\geq 1$, and $d_{-1}(0)=1$. From \cite[Lemma 2.3]{Ula} we get the equality
\begin{align*}
	\nu_{p}(d_{-1}(n))=\left\{\begin{array}{ll}
		0, & \text{ if } n=0, \\
		1, & \text{ if } n\geq 1.
	\end{array}\right.
\end{align*}

We will also need the following simple fact.

\begin{lem}\label{paryneglem1}
	Let $i\in\N_{\geq2}$ and $p\in\mathbb{P}_{\geq 3}$. Then
	\begin{align*}
		i\geq \nu_{p}(i)+2.
	\end{align*}
\end{lem}
\begin{proof}
	Let $i\in\N_{\geq 1}$. We always have $i>\nu_{p}(i)$, i.e., $i\geq \nu_{p}(i)+1$. Thus if $i<\nu_{p}(i)+2$, then $i=\nu_{p}(i)+1$. From the Bernoulli's inequality we get:
	\begin{align*}
		i\geq p^{\nu_{p}(i)}\geq 2^{\nu_{p}(i)}=(1+1)^{\nu_{p}(i)}\geq 1+\nu_{p}(i)=i.
	\end{align*}
	Hence, $p^{\nu_{p}(i)}=2^{\nu_{p}(i)}$, so $\nu_{p}(i)=0$ and thus $i=1$.
\end{proof}

We are ready to prove the following analogue of Theorem \ref{parythm1}.

\begin{thm}\label{parythm2}
	Let $m\in\N_{\geq 1}$. Then
	\begin{align*}
		d_{-m}(n)\equiv md_{-1}(n)\Mod{p^{\nu_{p}(m) +2}}.
	\end{align*}
	In particular,
	\begin{align*}
		\nu_{p}(d_{-m}(n))=\left\{\begin{array}{ll}
				0,           &  \text{ if } n=0, \\
				\nu_{p}(m)+1,&  \text{ if } n\geq 1.
			\end{array}\right.
	\end{align*}
\end{thm}
\begin{proof}
	Firstly, we prove the result for $m=p^{\alpha}$. Observe, that
	\begin{align*}
		H_{-p^{\alpha}}(x)= & \left(H_{-1}(x)\right)^{p^{\alpha}}=\sum_{n=0}^{\infty}\left(\sum_{j_{1}+\ldots +{j_{p^{\alpha}}}=n}d_{-1}(j_{1})\cdot\ldots\cdot d_{-1}(j_{p^{\alpha}})\right)x^{n} \\
		= & 1+ \sum_{n=1}^{\infty}\left[\sum_{i=1}^{p^{\alpha}}\left(\binom{p^{\alpha}}{i}\sum_{\substack{ j_{1}+\ldots +j_{i}=n \\ j_{1},\ldots ,j_{i}\geq 1 }}d_{-1}(j_{1})\cdot\ldots\cdot d_{-1}(j_{i})\right)\right]x^{n},
	\end{align*}
	because $d_{-1}(0)=1$. From Lemma \ref{paryneglem1} we have for $i\geq 2$:
	\begin{align*}
		\nu_{p}\left(\binom{p^{\alpha}}{i}\sum_{\substack{ j_{1}+\ldots +j_{i}=n \\ j_{1},\ldots ,j_{i}\geq 1 }}d_{-1}(j_{1})\cdot\ldots\cdot d_{-1}(j_{i})\right)\geq \alpha-\nu_{p}(i)+i\geq \alpha +2.
	\end{align*}
	Therefore,
	\begin{align*}
		H_{-p^{\alpha}}(x)\equiv & 1+ \sum_{n=1}^{\infty}p^{\alpha}d_{-1}(n)x^{n}\Mod{p^{\alpha +2}}.
	\end{align*}
	Hence, the result is true when $m$ is a power of $p$ or $m=1$. If $m=p^{\alpha}k$, where $p\nmid k$, then similarly as in the proof of Theorem \ref{parythm1} we get:
	\begin{align*}
		H_{-m}(x)= & \left(H_{-p^{\alpha}}(x)\right)^{k}=\sum_{n=0}^{\infty}\left(\sum_{j_{1}+\ldots +j_{k}=n}d_{-p^{\alpha}}(j_{1})\cdot\ldots\cdot d_{-p^{\alpha}}(j_{k})\right)x^{n} \\
		\equiv & \sum_{n=0}^{\infty}kd_{-p^{\alpha}}(n)x^{n}\equiv \sum_{n=0}^{\infty}kp^{\alpha}d_{-1}(n)x^{n}= \sum_{n=0}^{\infty}md_{-1}(n)x^{n} \Mod{p^{\alpha +2}}.
	\end{align*}
	The result follows.
\end{proof}

\begin{cor}\label{cor1'}
If $m\leq -1$, then $d_{m}(n)\neq 0$ for $n\in\N$.
\end{cor}
\begin{proof}
From our description of the $p$-adic valuation of $d_{m}(n)$ given in Theorem \ref{parythm2}, we see that $\nu_{p}(d_{m}(n))\neq +\infty$ for $n\in\N$. This is equivalent with the non-vanishing of $d_{m}(n)$ for $n\in\N$ and hence the result.
\end{proof}

As a very surprising corollary from Theorems \ref{parythm1} and \ref{parythm2} we get the next fact.

\begin{cor}\label{corsequences}
	If $p\in\mathbb{P}_{\geq 3}$ and $m\in\Z\setminus\{0\}$, then the sequence $\left(\frac{d_{m}(n)}{pm}\Mod p\right)_{n\in\N_{+}}$ depends only on the sign of $m$.
\end{cor}

\section{$p$-automaticity of the sequence $\left(\frac{d_{m}(n)}{pm}\Mod p\right)_{n\in\N_{+}}$}\label{sec4}

For $p\in\mathbb{P}_{\geq 3}$ and $n\in\N_{+}$, let us denote $y_{p}(n):=\frac{d_{m}(n)}{pm}\Mod p$ if $m>0$, and $z_{p}(n):=\left(-\frac{d_{m}(n)}{pm}\right){\Mod p}$ if $m<0$. In this section we will study the sequences $(y_{p}(n))_{n\in\N_{+}}, (z_{p}(n))_{n\in\N_{+}}$ more closely. In particular we prove $p$-automaticity of our sequences.

%In order to do so, we introduce the notion of an {\it $k$-automatic sequence}. More precisely, let $k\in\N_{\geq 2}$ be given. %We say that the sequence ${\bf \eps}=(\eps_{n})_{n\in\N}$ is {\it $k$-automatic} if and only if the following set
%$$
%K_{k}({\bf \eps})=\{(\eps_{k^{i}n+j})_{n\in\N}:\;i\in\N\;\mbox{and}\;0\leq j<k^{i}\},
%$$
%called the $k$-kernel of ${\bf \eps}$, is finite. If the $k$-kernel of the sequence ${\bf \eps}$ is infinite but finitely %generated then we say that our sequence is {\it $k$-regular}.

%Let us also recall that one of equivalent conditions for automaticity, is the following result of Christol: let %$p\in\mathbb{P}$ and $q$ be a power of $p$, and ${\bf a}=(a_{n})_{n\in\N}$ be a sequence over $\mathbb{F}_{q}$. Then ${\bf a}$ %is $p$-automatic if and only if the formal power series $\sum_{n=0}^{\infty}a_{n}x^{n}$ is algebraic over $\mathbb{F}_{q}(x)$. %The proof can be found in \cite{Chr} or in recent monograph \cite[Theorem 12.2.5]{AllSh}.

Let
\begin{align*}
	Y_{p}(x)= & \sum_{n=0}^{\infty}y_{p}(n+1)x^{n}, \\
	Z_{p}(x)= & \sum_{n=0}^{\infty}z_{p}(n+1)x^{n},
\end{align*}
i.e., $Y_{p}$ and $Z_{p}$ are the generating functions of the sequences $(y_{p}(n))_{n\in\N_{+}}$ and $(z_{p}(n))_{n\in\N_{+}}$, respectively.

Now we are ready to state the main result of this section.

\begin{thm}\label{automaticthm1}
	Both sequences, $(y_{p}(n))_{n\in\N_{+}}$ and $(z_{p}(n))_{n\in\N_{+}}$, are $p$-automatic.
\end{thm}
\begin{proof}
	By Christol's Theorem, it is enough to show, that $Y_{p}$ and $Z_{p}$ are both algebraic over $\mathbb{F}_{p}(x)$. First we consider the series $Y_{p}$. Let us denote
	\begin{align*}
		W_{p}(x):=\sum_{n=1}^{p-1}(n^{-1}\Mod{p})x^{n-1}.
	\end{align*}
	In the proof of Theorem \ref{parythm1} it was shown, that the sequence $(y_{p}(n))_{n=1}^{\infty}$ satisfies the following recurrence relations:
	\begin{align*}
		\left\{\begin{array}{ll}
			y_{p}(n)=n^{-1}\Mod{p}, & \text{ if } 1\leq n\leq p-1, \\
			y_{p}(n)=y_{p}(n\Mod{p}), & \text{ if } p\nmid n, \\
			y_{p}(pn)=y_{p}(n).
		\end{array}\right.
	\end{align*}
	Therefore,
	\begin{align*}
		Y_{p}(x)= & \sum_{n=1}^{\infty}y_{p}(pn)x^{pn-1}+\sum_{n=0}^{\infty}\left(y_{p}(1)+y_{p}(2)x+\ldots y_{p}(p-1)x^{p-2}\right)x^{pn} \\
		= & x^{p-1}\sum_{n=1}^{\infty}y_{p}(n)x^{p(n-1)}+W_{p}(x)\sum_{n=0}^{\infty}x^{pn} \\
		= & x^{p-1}\sum_{n=0}^{\infty}y_{p}(n+1)x^{pn}+\frac{W_{p}(x)}{1-x^{p}} \\
		= & x^{p-1}Y_{p}(x^{p})+\frac{W_{p}(x)}{1-x^{p}}.
	\end{align*}
	Equivalently,
	\begin{align}\label{equY}
		x^{p-1}Y_{p}(x)^{p}-Y_{p}(x)+\frac{W_{p}(x)}{1-x^{p}}=0,
	\end{align}
	so $Y_{p}$ is indeed algebraic over $\mathbb{F}_{p}(x)$.
	
	For the proof of the case of the sequence $Z_{p}$, notice the following relations:
	\begin{align*}
		\Delta_{p}(x)= & \prod_{n=0}^{\infty}\left(1-x^{p^{n}}\right)^{p-1}=(1-x)^{p-1}\Delta_{p}(x^{p}), \\
		H_{-1}(x)= & (1-x)\Delta_{p}(x)=(1-x)^{p}\Delta_{p}(x^{p})=\frac{(1-x)^{p}}{1-x^{p}}H_{-1}(x^{p}), \\
		Z_{p}(x)\equiv & \frac{1}{px}\left(H_{-1}(x)-1\right) \Mod{p}.
	\end{align*}
	The last congruence follows from the fact that $z_{p}(n)\equiv \frac{d_{-1}(n)}{p}\Mod{p}$. Thus (we consider the equalities below as equalities over $\mathbb{F}_{p}$, i.e., modulo $p$)
	\begin{align*}
		Z_{p}(x)= & \frac{1}{px}\left(H_{-1}(x)-1\right)=\frac{1}{px}\left(\frac{(1-x)^{p}}{1-x^{p}}H_{-1}(x^{p})-1\right) \\
		= & \frac{1}{px}\left(\frac{(1-x)^{p}}{1-x^{p}}\left(H_{-1}(x^{p})-1\right)+\frac{(1-x)^{p}}{1-x^{p}}-1\right) \\
		= & x^{p-1}\frac{(1-x)^{p}}{1-x^{p}}\frac{1}{px^{p}}\left(H_{-1}(x^{p})-1\right)+\frac{1}{px}\frac{(1-x)^{p}-(1-x^{p})}{1-x^{p}} \\
		= & x^{p-1}\frac{(1-x)^{p}}{1-x^{p}}Z_{p}(x^{p})+\frac{1}{1-x^{p}}\left(\sum_{n=1}^{p-1}(-1)^{n}\frac{1}{p}\binom{p}{n}x^{n-1}\right) \\
		= & x^{p-1}Z_{p}(x^{p})-\frac{W_{p}(x)}{1-x^{p}},
	\end{align*}
	because for each $1\leq n\leq p-1$:
	\begin{align*}
		(-1)^{n}\frac{1}{p}\binom{p}{n}=(-1)^{n}\frac{(p-1)\ldots (p-n+1)}{n!}\equiv (-1)^{n}\frac{(-1)^{n-1}(n-1)!}{n!}=-\frac{1}{n}\Mod{p}.
	\end{align*}
	We can write the equation for $Z_{p}$ in an equivalent form as
	\begin{align}\label{equZ}
		x^{p-1}Z_{p}(x)^{p}-Z_{p}(x)-\frac{W_{p}(x)}{1-x^{p}}=0,
	\end{align}
	so $Z_{p}$ is also algebraic over $\mathbb{F}_{p}(x)$.
\end{proof}

The explicit form of the equations satisfied by $Y_{p}$ and $Z_{p}$ allow us to show the strong relation between sequences $(y_{p}(n))_{n\in\N_{+}}$ and $(z_{p}(n))_{n\in\N_{+}}$.

\begin{cor}\label{parycor1}
	If $p\in\mathbb{P}_{\geq 3}$ and $n\in\N_{\geq 1}$, then:
	\begin{align*}
		y_{p}(n)+z_{p}(n)\equiv 0\Mod{p}.
	\end{align*}
\end{cor}
\begin{proof}
	Summing equations (\ref{equY}) and (\ref{equZ}) we get
	\begin{align*}
		\left(Y_{p}(x)+Z_{p}(x)\right)\left(x^{p-1}\left(Y_{p}(x)+Z_{p}(x)\right)^{p-1}-1\right)\equiv 0\Mod{p}.
	\end{align*}
Because the $0$th term in the second factor is not divisible by $p$ we need to have
	\begin{align*}
		Y_{p}(x)+Z_{p}(x)\equiv 0\Mod{p},
	\end{align*}
	and this is equivalent to the statement.
\end{proof}

%It is worth to note, that a sequence is $p$-automatic if and only if there exists an automaton for it. Roughly speaking, %automaton is a procedure that allow us to compute for each $n\in\N$ the $n$-th member of a sequence by using only the digits %in the base-$p$ representation of $n$. For more precise definition see \cite{AllSh1}. We can either use the convention to read %the digits of the base-$p$ representation of $n$ from the highest power of $p$ that appear in the representation to the lowest %one (that is, from the right to the left), or conversely, that is, from the lowest (i.e., from the $0$'th power) to the %highest one (that is, from the left to the right). It is known, that a sequence $\bf{a}$ is $p$-automatic in the first %convention if and only if it is $p$-automatic in the second one, but automatons generating $\bf{a}$ may be different depending %on convention.
In the case of sequences $(y_{p}(n))_{n\in\N_{+}}$ and $(z_{p}(n))_{n\in\N_{+}}$ we can write the automatons down in both conventions (recall the discussion on the end of Section \ref{sec2}). Now we show them for the first of these sequences. In the case of the second one, it is enough to switch $n^{-1}$ into $-n^{-1}$ for all $n=1, \ldots, p-1$.
The automaton generating $(y_{p}(n))_{n\in\N_{+}}$ when we read digits from the left to the right is presented below.

\bigskip

\begin{center}
\begin{tikzpicture}[->, shorten >= 1pt, node distance=2.5 cm, on grid, auto]
\node[state, initial] (c_0) {$-$};
\node[state, inner sep=1pt] (c_1) [above right=2.5cm and 5cm of c_0] {$1^{-1}$};
\node[state, inner sep=1pt] (c_2) [above right=1cm and 5cm of c_0] {$2^{-1}$};
\node[inner sep=1pt] (c_3) [below right=0cm and 5cm of c_0] {$\vdots$};
\node[state, inner sep=1pt] (c_4) [below right=1.5cm and 5cm of c_0] {$(p-1)^{-1}$};
\path[->]
(c_0) edge [left, above] node {$1$} (c_1)
edge [left, above] node {$2$} (c_2)
edge [left, below] node {$p-1$} (c_4)
edge [loop above] node {$0$} (c_0)
(c_1) edge [loop right] node {$0,\ldots, p-1$} (c_1)
(c_2) edge [loop right] node {$0,\ldots, p-1$} (c_2)
(c_4) edge [loop right] node {$0,\ldots, p-1$} (c_4);	
\end{tikzpicture}
\end{center}
\begin{center}
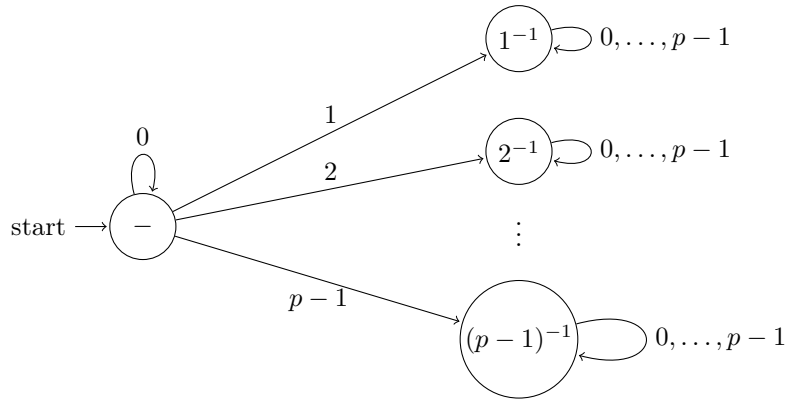
Figure 3. Automaton generating $(y_{p}(n))_{n\in\N}$ when reading digits of $n$ from the left to the right.
\end{center}
\bigskip

On the other side, if we read the digits in the opposite direction we present the corresponding automaton generating $(y_{p}(n))_{n\in\N}$ below.

\bigskip

\begin{center}
\begin{tikzpicture}[->, shorten >= 1pt, node distance=2.5 cm, on grid, auto]
\node[state, initial] (c_0) {$-$};
\node[state, inner sep=1pt] (c_1) [above right=3cm and 5cm of c_0] {$1^{-1}$};
\node[state, inner sep=1pt] (c_2) [above right=1cm and 5cm of c_0] {$2^{-1}$};
\node[inner sep=1pt] (c_3) [below right=0.7cm and 5cm of c_0] {$\vdots$};
\node[state, inner sep=1pt] (c_4) [below right=3cm and 5.2cm of c_0] {$(p-1)^{-1}$};
\path[->]
(c_0) edge [left, above] node {$1$} (c_1)
edge [left, above] node {$2$} (c_2)
edge [left, below] node {$p-1$\ \ \ \ \ } (c_4)
(c_1) edge [loop above] node {$0, 1$} (c_1)
edge [bend left=20] node {$2$} (c_2)
edge [bend left=85] node {$p-1$} (c_4)
(c_2) edge [loop right] node {$0, 2$} (c_2)
edge [bend left=20] node {$1$} (c_1)
edge [bend left=40] node {$p-1$} (c_4)
(c_4) edge [loop below] node {$0, p-1$} (c_4)
edge [bend right=25] node {$2$} (c_2)
edge [bend right=65] node {$1$} (c_1);	
\end{tikzpicture}
\end{center}
\begin{center}
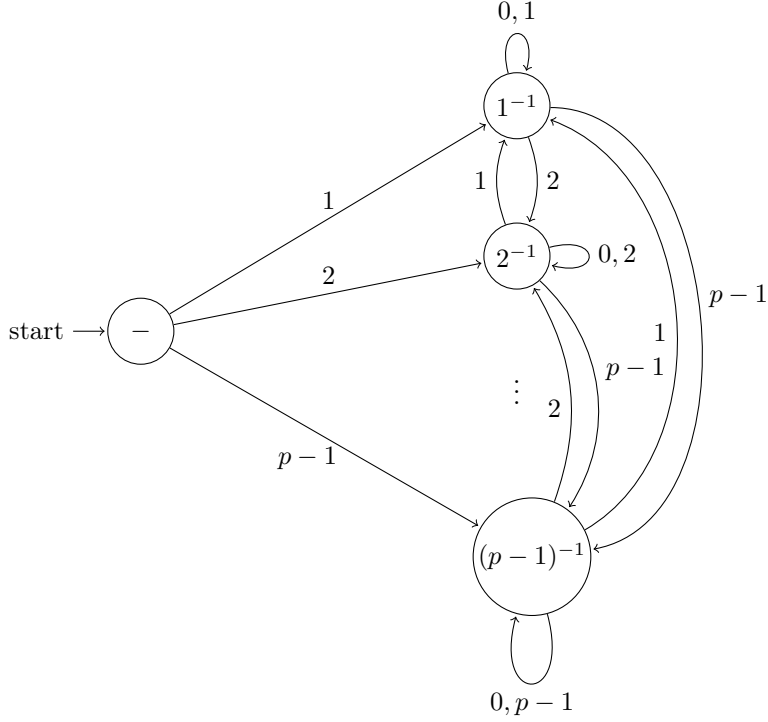
Figure 4. Automaton generating $(y_{p}(n))_{n\in\N}$ when reading digits of $n$ from the right to the left.
\end{center}

\bigskip

More precisely, in the second automaton each arrow with number $0$ forms a loop, and each arrow with a number $n\in\{1,\ldots ,p-1\}$ goes to the state with $n^{-1}$.

\bigskip

Let us note that in the case $p\geq 3$ the sequence $(\nu_{p}(d_{m}(n)))_{n\in\N}$ is $p$-automatic, which is an immediate consequence of the fact that it is ultimately constant.

\bigskip

%Go back to the sequences $(y_{p}(n))_{n\in\N_{+}}$ and $(z_{p}(n))_{n\in\N_{+}}$.
Corollary \ref{parycor1} can be used to find a relation between numbers $d_{m}(n)$ and $d_{-m}(n)$, and generalize some results obtained previously only for positive $m$. We have the following

\begin{thm}
	For each $p\in\mathbb{P}_{\geq 3}$, $m\in\N_{\geq 1}$ and $n\in\N$ we have
	\begin{align*}
		d_{-m}(n)\equiv -d_{m}(n)\Mod{p^{\nu_{p}(m)+2}}.
	\end{align*}
	In particular,
	\begin{align*}
		d_{-m}(pn)\equiv d_{-m}(n)\Mod{p^{\nu_{p}(m)+2}}.
	\end{align*}
\end{thm}
\begin{proof}
	The first part follows from Theorems \ref{parythm1} and \ref{parythm2}, and Corollary \ref{parycor1}. Indeed,
	\begin{align*}
		d_{-m}(n)\equiv pmz_{p}(n)\equiv -pmy_{p}(n)\equiv -d_{m}(n)\Mod{p^{\nu_{p}(m)+2}}.
	\end{align*}
	The second part is now an immediate consequence of Lemma \ref{parylem2}.
\end{proof}

We group the fact above, Theorems \ref{parythm1} and \ref{parythm2}, and Corollary \ref{corsequences} and simply get the following general result.

\begin{cor}\label{parycor2}
	Let $p\in\mathbb{P}_{\geq 3}$ and $m\in\Z\setminus\{0\}$. Let $n\in\N_{\geq 1}$ and $n=n_{s}p^{s}+\ldots +n_{t}p^{t}$ be its base--$p$ representation, where $s\leq t$ and $n_{s}\neq 0$. Then
	\begin{align*}
		d_{m}(n)\equiv pm\left(n_{s}^{-1}\Mod p\right)\Mod{p^{\nu_{p}(m)+2}}.
	\end{align*}
	In particular,
	\begin{align*}
	\nu_{p}(d_{m}(n))=\left\{\begin{array}{ll}
	0, & \text{ if } n=0, \\
	\nu_{p}(m)+1, & \text{ if } n\geq 1,
	\end{array}
	\right.
	\end{align*}
	and the sequence $\left(\frac{d_{m}(n)}{pm}\Mod{p}\right)_{n=1}^{\infty}$ does not depend on $m$.
\end{cor}

We can apply Corollary \ref{parycor2} to the problem that was considered in \cite{Ula}. Let us define for $p\in\mathbb{P}_{\geq 3}$, $m\in\N_{\geq 1}$ a sequence $(A_{p,m}(n))_{n\in\N}$ as
\begin{align*}
	\prod_{n=0}^{\infty}\frac{1}{(1-x^{p^{n}})^{m}}=\sum_{n=0}^{\infty}A_{p,m}(n)x^{n}.
\end{align*}
The number $A_{p,m}(n)$ counts the number of $p$-ary representations of $n$ such that each part takes one among $m$ colors. The sequence $(A_{p,m}(n))_{n\in\N}$ was thorough studied in \cite{Ula}. One of main results of the cited paper states that for each $\alpha\in\N_{\geq 1}$ and $n\geq p^{\alpha}$ the following equality holds:
\begin{align*}
	\nu_{p}\left(A_{p,(p-1)(p^{\alpha}-1)}(n)\right)=1.
\end{align*}

Using our previous findings we are able to improve the above equality to the following

\begin{thm}
	Let $p\in\mathbb{P}_{\geq 3}$, $\alpha\in\N_{\geq 1}$ and $n\geq p^{\alpha}$. Let $n=\sum_{j=s}^{t}n_{j}p^{j}$, where $s\leq t$ and $n_{s}\neq 0$, be the base-$p$ representation of $n$. Then
	\begin{align*}
		A_{p,(p-1)(p^{\alpha}-1)}(n)\equiv -p(n_{s}^{-1}\Mod{p})\Mod{p^{2}}.
	\end{align*}
\end{thm}
\begin{proof}
	The relation on the very end of the proof of Theorem 3.1 from \cite{Ula} implies that
	\begin{align*}
		A_{p,(p-1)(p^{\alpha}-1)}(n)\equiv \left(D_{p}(n)-D_{p}(n-1)\right)\prod_{j=0}^{\alpha -1}(-1)^{n_{j}}\binom{p-1}{n_{j}}\Mod{p^{2}}
	\end{align*}
	%(we assumed that $n\geq p^{\alpha}$, so $s\leq \alpha$. If $s=\alpha$, we define the product above to be equal to $1$).
	(we assume that $n_{j}=0$ for $j<s$). Corollary \ref{parycor2} gives
	\begin{align*}
		A_{p,(p-1)(p^{\alpha}-1)}(n)\equiv & d_{-1}(n)\prod_{j=0}^{\alpha -1}(-1)^{n_{j}}\binom{p-1}{n_{j}}\equiv -p(n_{s}^{-1}\Mod{p})\prod_{j=0}^{\alpha -1}(-1)^{n_{j}}\binom{p-1}{n_{j}}\Mod{p^{2}}.
	\end{align*}
	 For the end of the proof it is enough to observe that for each $j\in\{0,\ldots ,\alpha -1\}$:
	 \begin{align*}
	 	(-1)^{n_{j}}\binom{p-1}{n_{j}}= & (-1)^{n_{j}}\frac{(p-1)\ldots (p-n_{j})}{n_{j}!}\equiv (-1)^{n_{j}}\frac{(-1)^{n_{j}}n_{j}!}{n_{j}!}=1 \Mod{p}.
	 \end{align*}
\end{proof}

We finish this section with the following

\begin{thm}\label{transres}
Let us consider the sequences $(y_{p}(n))_{n\in\N_{+}}$, $(z_{p}(n))_{n\in\N_{+}}$ as the sequences with terms in the set $\{1,\ldots,p-1\}$ and the corresponding ordinary generating functions $Y_{p}(x)$, $Z_{p}(x)$ as the power series in $\C[[x]]$. Then both $Y_{p}$, $Z_{p}$ are analytic in the circle $|x|<1$ and are transcendental over $\C(x)$.
\end{thm}
\begin{proof}
%Because $y_{p}(n), z_{p}(n)\in\{1,\ldots,p-1\}$ and $z_{p}(n)=-y_{p}(n)$ (treated as elements in $\Z$, then it is clear that the series $Y_{p}(x), Z_{p}(x)$ are absolutely convergent in the set $\{x\in\C:\;|x|<1\}$ and thus define analytic functions. Moreover, from the proof of Theorem \ref{automaticthm1} we know that the series $Y_{p}$, $Z_{p}$ satisfy functional equations
%\begin{equation*}
%Y_{p}(x)=x^{p-1}Y_{p}(x^{p})+\frac{W_{p}(x)}{1-x^{p}},\quad\quad Z_{p}(x)=x^{p-1}Z_{p}(x^{p})-\frac{W_{p}(x)}{1-x^{p}},
%\end{equation*}
%where $W_{p}(x)=\sum_{n=1}^{p-1}(n^{-1}\mod{p})x^{n-1}$.

Because $y_{p}(n), z_{p}(n)\in\{1,\ldots,p-1\}$ (treated as elements in $\Z$), then it is clear that the series $Y_{p}(x)$ and $Z_{p}(x)$ are absolutely convergent in the set $\{x\in\C:\;|x|<1\}$ and thus define analytic functions. From Corollary \ref{parycor1} and the fact that $0<y_{p}(n),z_{p}(n)<p$, we get that $z_{p}(n)=p-y_{p}(n)$ for all $n\in\N_{+}$. Thus simply
\begin{align*}
	Z_{p}(x)=\frac{p}{1-x}-Y_{p}(x).
\end{align*}
Therefore, in order to get the transcendence of the functions $Y_{p}(x)$ and $Z_{p}(x)$, it is enough to prove it only for $Y_{p}(x)$.

In the proof of Theorem \ref{automaticthm1} we obtained the functional equation for the function $Y_{p}(x)$:
\begin{equation*}
Y_{p}(x)=x^{p-1}Y_{p}(x^{p})+\frac{W_{p}(x)}{1-x^{p}},
\end{equation*}
where $W_{p}(x)=\sum_{n=1}^{p-1}(n^{-1}\mod{p})x^{n-1}$. This equation was proved in fact over $\C$, not only over $\F_{p}$ (because all equalities in this case were 'true' equalities, i.e., not equalities modulo $p$).

The transcendence of $Y_{p}$ will be consequence of a general result of Nishioka \cite{Nis}. This result says that a power series, say $f$, with rational coefficients, which defines an analytic function in some neighborhood of zero and satisfying functional equation of the form $f(x)=A(x)+B(x)f(x^{m})$ for some $m\in\N_{\geq 2}$ and functions $A, B\in\C(x)$, is either rational of transcendental. In the light of this remark we see, that in order to get transcendence of $Y_{p}$ it is enough to prove that our function is not rational. %We present the proof only in the case of the function $Y_{p}(x)$. Indeed, from Corollary \ref{parycor1} we have $Z_{p}(x)=-Y_{p}(x)$.

Let us suppose that $Y_{p}(x)$ is rational, i.e., $Y_{p}(x)=P(x)/Q(x)$ for some $P, Q\in\C[x]$ with $\gcd(P(x),Q(x))=1$. Putting the expression for $Y_{p}$ into the functional equation and clearing the denominators we get the equation:
\begin{equation}\label{eqrat1}
(1-x^{p})P(x)Q(x^{p})=(1-x^{p})x^{p-1}P(x^{p})Q(x)+Q(x)Q(x^{p})W_{p}(x).
\end{equation}
From the above equation we get $Q(x^{p})\mid (1-x^{p})x^{p-1}P(x^{p})Q(x)$. However, the co-primality of the polynomials $P(x), Q(x)$ implies co-primality of the polynomials $P(x^{p}), Q(x^{p})$ and thus we get
\begin{align*}
	Q(x^{p})\mid (1-x^{p})x^{p-1}Q(x).
\end{align*}
Consequently $p\op{deg}Q\leq 2p-1+\op{deg}Q$ and thus $\op{deg}Q\leq 2+\frac{1}{p-1}$. We thus see that the degree of the polynomial $Q$ is bounded by 2. %We want to show that actually, it has to be equal to $2$.

Let us put $x=1$ in (\ref{eqrat1}). We get that $Q(1)^{2}W_{p}(1)=0$, but $W_{p}(1)>0$, so $Q(1)=0$. In particular, $\deg Q\geq 1$ and we can write $Q(x)=(1-x)S(x)$. Plug it into (\ref{eqrat1}) and get after dividing by $1-x$:
\begin{align}\label{eqrat2}
	\frac{1-x^{p}}{1-x}P(x)S(x^{p})=x^{p-1}P(x^{p})S(x)+S(x)S(x^{p})W_{p}(x).
\end{align}
If $\deg Q=1$, then $S$ is a constant, say $S(x)=S_{0}\neq 0$. Equation (\ref{eqrat2}) implies:
\begin{align*}
	W_{p}(x)=\frac{1}{S_{0}}\left(\frac{1-x^{p}}{1-x}P(x)-x^{p-1}P(x^{p})\right).
\end{align*}
Observe, that $\deg W_{p}(x)=p-2$. If $\deg P\geq 1$, then the degree of the polynomial of the right hand side of the last equation is equal to $\deg (x^{p-1}P(x^{p}))\geq 2p-1>p-2$, a contradiction. Hence, $\deg P=0$, so we can write $P(x)=P_{0}\neq 0$. Then we get from the last equation:
\begin{align*}
	W_{p}(x)=\frac{P_{0}}{S_{0}}\left(1+x+\ldots +x^{p-2}\right),
\end{align*}
that is again a contradiction, since not all coefficients of $W_{p}(x)$ are equal to each other. Thus we may assume that $\deg Q=2$ and $\deg S=1$.

Let us write $S(x)=S_{1}x+S_{0}$. From (\ref{eqrat2}) we get $S(x^{p})\mid x^{p-1}P(x^{p})S(x)$, but $S(x^{p})$ and $P(x^{p})$ are co-prime, so $S(x^{p})\mid x^{p-1}S(x)$. Thus there exists a polynomial $T(x)$ such that $S(x^{p})T(x)=x^{p-1}S(x)$, and by comparing the degrees we get $\deg T=0$. Comparing the coefficients near $x^{p-1}$ implies $S_{0}=0$, i.e., $S(x)=S_{1}x$. Equation (\ref{eqrat2}) implies:
\begin{align*}
	W_{p}(x)=\frac{1}{S_{1}x}\left(\frac{1-x^{p}}{1-x}P(x)-P(x^{p})\right).
\end{align*}
If $\deg P\geq 2$, then $\deg\left(\frac{1}{S_{1}x}\left(\frac{1-x^{p}}{1-x}P(x)-P(x^{p})\right)\right)=p\deg P-1>p-2$. Hence, $\deg P\leq 1$. Write $P(x)=P_{1}x+P_{0}$. Thus
\begin{align*}
	W_{p}(x)=\frac{1}{S_{1}x}\left(\frac{1-x^{p}}{1-x}(P_{1}x+P_{0})-(P_{1}x^{p}+P_{0})\right)=\frac{P_{1}+P_{0}}{S_{1}}\left(1+x+\ldots +x^{p-2}\right),
\end{align*}
that is a contradiction.

Summing up our discussion: we proved that there are no polynomials $P, Q\in\C[x]$ satisfying the equation (\ref{eqrat1}) and thus the function $Y_{p}(x)$ can not be rational.
\end{proof}

\begin{rem}
{\rm In the above proof, instead of use result of Nishioka, we could use classical result of Fatou: {\it if a power series $\sum_{n=0}^{\infty}a_{n}x^{n}$ with integer coefficients converges inside the unit disk, then it is either rational or transcendental over $\C(x)$} \cite{Fa}. However, it is clear that the burden of the proof lies in the proof of irrationality of $Y_{p}(x)$.
}
\end{rem}

An immediate consequence of our result is the following.

\begin{cor}
The sequences $(y_{p}(n))_{n\in\N_{+}}, (z_{p}(n))_{n\in\N_{+}}$ are not periodic.
\end{cor}
\begin{proof}
It is clear that the ordinary generating function of a periodic sequence is rational. However, in the theorem above we proved that the ordinary generating functions of our sequences are transcendental.
\end{proof}

\begin{rem}
{\rm One can investigate further properties of the sequence $(y_{p}(n))_{n\in\N_{+}}$. We note only one property. More precisely, one can easily prove the following summation formula
\begin{align*}
\sum_{k=1}^{p^{n}-1}y_{p}(k)=\frac{1}{2}p(p^{n}-1).
\end{align*}
}
\end{rem}

\section{Further congruences for $d_{m}(pn)-d_{m}(n)$}\label{sec5}

Observe, that Theorem \ref{parythm1} can be viewed as a generalization of Lemma \ref{parylem1}. Therefore, it is natural to use this more general fact in order to generalize some of results that we have previously obtained. In the sequel, we will need the following lemma, that is a generalization of Wolstenholme's theorem \cite{Wol}, which says, that for $p\in\mathbb{P}_{\geq 5}$ the following congruence holds
\begin{align*}
	\binom{pm}{pi}\equiv \binom{m}{i}\Mod{p^{3}}.
\end{align*}
The Wolstenholme's theorem is equivalent to the following pair of the congruences
\begin{align*}
	\sum_{\gamma =1}^{p-1}\frac{1}{\gamma}\equiv0\Mod{p^{2}} \ \ \ \ \ \ \ \ \ \ \ \text{ and } \ \ \ \ \ \ \ \ \ \ \
	\sum_{\gamma =1}^{p-1}\frac{1}{\gamma^{2}}\equiv0\Mod{p},
\end{align*}
that will be used in the proof.

\begin{lem}\label{parylem4}
	Let $m\in\N_{\geq 1}$, $i\in\{0,\ldots ,m\}$, $p\in\mathbb{P}_{\geq 3}$. Then
	\begin{align*}
		\binom{pm}{pi}\equiv \binom{m}{i}\Mod{p^{\nu_{p}(m)+\nu_{p}\left(\binom{m}{i}\right)+3-\chi}},
	\end{align*}
	where $\chi =[p=3]-\left[p=3,\ \nu_{3}(m)\geq 1,\ \nu_{3}\left(\binom{m}{i}\right)=0\right]$.
\end{lem}
\begin{proof}
	The statement is obvious if $i=0$ or $i=m$. Assume $1\leq i\leq m-1$ and denote $\alpha:=\nu_{p}(m)$. Then
	\begin{align*}
		\binom{pm}{pi}= & \frac{\prod_{j=0}^{pi-1}(pm-j)}{\prod_{j=1}^{pi}j}=\prod_{j=1}^{i}\frac{pm-pj+p}{pj}\prod_{j=0}^{i-1}\prod_{\beta =1}^{p-1}\frac{pm-pj-\beta}{pj+\beta}=\binom{m}{i}\prod_{j=0}^{i-1}\prod_{\beta =1}^{p-1}\frac{pm-(pj+\beta)}{pj+\beta}.
	\end{align*}
	In order to finish the proof, it is enough to show that for each $j$:
	\begin{align*}
		\prod_{\beta =1}^{p-1}(pm-(pj+\beta))\equiv \prod_{\beta =1}^{p-1}(pj+\beta)\Mod{p^{\alpha +3}}.
	\end{align*}
	We have
	\begin{align*}
		\prod_{\beta =1}^{p-1}(pm-(pj+ & \beta)) - \prod_{\beta =1}^{p-1}(pj+\beta) \\
		\equiv & p^{2}m^{2}\sum_{\gamma_{1}<\gamma_{2}}\prod_{\substack{\beta=1 \\ \beta\neq \gamma_{1},\gamma_{2}}}^{p-1}(pj+\beta)-pm\sum_{\gamma =1}^{p-1}\prod_{\substack{\beta=1 \\ \beta\neq \gamma}}^{p-1}(pj+\beta) +\prod_{\beta =1}^{p-1}(pj+\beta)-\prod_{\beta =1}^{p-1}(pj+\beta) \\
		= & \prod_{\beta =1}^{p-1}(pj+\beta)\left(p^{2}m^{2}\sum_{\gamma_{1}<\gamma_{2}}\frac{1}{(pj+\gamma_{1})(pj+\gamma_{2})}-pm\sum_{\gamma =1}^{p-1}\frac{1}{pj+\gamma}\right) \Mod{p^{\alpha +3}}.		
	\end{align*}
	If $p=3$ and $\nu_{3}(m)$ and $\nu_{3}\left(\binom{m}{i}\right)$ are arbitrary, we are in fact interested in the congruence modulo $p^{\alpha+2}$ instead of $p^{\alpha +3}$. In that case, the quantity in the brackets simplifies to
	\begin{align*}
		-pm\sum_{\gamma =1}^{p-1}\frac{1}{pj+\gamma}\equiv -pm\sum_{\gamma =1}^{p-1}\frac{1}{\gamma}\equiv -pm\sum_{\gamma =1}^{p-1}\gamma=-p^{2}m\frac{p-1}{2}\equiv 0\Mod{p^{\alpha +2}},
	\end{align*}
	so the statement holds. When $p=3$ and additionally $\nu_{3}(m)\geq 1$ and $\nu_{3}\left(\binom{m}{i}\right)=0$, then $3^{\alpha+3}\mid 3^{2}m^{2}$, so the expression in the square brackets simplifies in the same way as before. Hence,
	\begin{align*}
		\prod_{\beta =1}^{3-1}(3m-(3j+\beta)) - \prod_{\beta =1}^{3-1}(3j+\beta)= & (3m-(3j+1))(3m-(3j+2))-(3j+1)(3j+2) \\
		\equiv & -3m(3j+1)(3j+2)\left(\frac{1}{3j+1}+\frac{1}{3j+2}\right) \\
		= & -3m(3j+1)(3j+2)\frac{6j+3}{(3j+1)(3j+2)} \\
		= &-9m(2j+1)\equiv 9m(j-1) \Mod{p^{\alpha +3}}.
	\end{align*}
	Therefore,
	\begin{align*}
		(3m-(3j+1))(3m-(3j+2))\equiv (3j+1)(3j+2)+9m(j-1)\Mod{p^{\alpha +3}}.
	\end{align*}
	The conditions $3\mid m$ and $3\nmid \binom{m}{i}$ imply $3\mid i$. Thus
	\begin{align*}
		\binom{3m}{3i}\equiv & \binom{m}{i}\prod_{j=0}^{i-1}\frac{(3m-(3j+1))(3m-(3j+2))}{(3j+1)(3j+2)}
		\equiv  \binom{m}{i}\prod_{j=0}^{i-1}\left(1+\frac{9m(j-1)}{(3j+1)(3j+2)}\right) \\
		\equiv & \binom{m}{i}\prod_{j=0}^{i-1}\left(1-9m(j-1)\right)
		\equiv  \binom{m}{i}\left(1-9m\sum_{j=0}^{i-1}(j-1)\right) \\
		\equiv & \binom{m}{i}\left(1-9m\left(\frac{i(i-1)}{2}-i\right)\right)
		\equiv  \binom{m}{i} \Mod{p^{\alpha +3}}.
	\end{align*}
	The proof of this case is complete.
	
	Assume that $p\geq 5$. Wolstenholme's Theorem implies:
	\begin{align*}
		\sum_{\gamma =1}^{p-1}\frac{1}{pj+\gamma}= & \frac{1}{2}\sum_{\gamma =1}^{p-1}\left(\frac{1}{pj+\gamma}+\frac{1}{pj+p-\gamma}\right) \\
		= & \frac{p}{2}\sum_{\gamma =1}^{p-1}\frac{2j+1}{(pj+\gamma)(pj+p-\gamma)}\equiv -(2j+1)\frac{p}{2}\sum_{\gamma =1}^{p-1}\frac{1}{\gamma^{2}}\equiv 0\Mod{p^{2}},
	\end{align*}
	and
	\begin{align*}
		\sum_{\gamma =1}^{p-1}\frac{1}{(pj+\gamma)^{2}}\equiv\sum_{\gamma =1}^{p-1}\frac{1}{\gamma^{2}}\equiv 0\Mod{p}.
	\end{align*}
	Therefore,
	\begin{align*}
		\sum_{\gamma_{1}<\gamma_{2}}\frac{1}{(pj+\gamma_{1})(pj+\gamma_{2})}=\frac{1}{2}\left[\left(\sum_{\gamma =1}^{p-1}\frac{1}{pj+\gamma}\right)^{2}-\sum_{\gamma =1}^{p-1}\frac{1}{(pj+\gamma )^{2}}\right]\equiv 0\Mod{p}.
	\end{align*}
	Hence, indeed
	\begin{align*}
		\prod_{\beta =1}^{p-1}(pm-(pj+\beta))\equiv \prod_{\beta =1}^{p-1}(pj+\beta)\Mod{p^{\alpha +3}}
	\end{align*}
	and the proof is finished.
\end{proof}

\begin{rem}
	{\rm For $i=1$, it is possible to strengthen the result. Indeed, using the techniques from the proof of Lemma \ref{modp} we get
	\begin{align*}
	\binom{pm}{p}-m= & m\left(\frac{(pm-1)\ldots (pm-p+1)}{(p-1)!}-1\right) \\
	= & \frac{m}{(p-1)!}\left((p(m-1)+p-1)(p(m-1)+p-2)\ldots (p(m-1)+1)-(p-1)!\right) \\
	\equiv & m\left(p^{2}(m-1)^{2}\sum_{\gamma_{1}<\gamma_{2}}\frac{1}{\gamma_{1}\gamma_{2}}+p(m-1)\sum_{\gamma =1}^{p-1}\frac{1}{\gamma}\right) \Mod{p^{\nu_{p}(m)+\nu_{p}(m-1)+3-\chi}}.
	\end{align*}
	We can repeat the proof of the previous fact and get result modulo %$p$ to the power
	$p^{\nu_{p}(m)+\nu_{p}(m-1)+3-\chi}$. Observe, that at most one of the number $\nu_{p}(m)$ and $\nu_{p}(m-1)$ may be non-zero. Hence, if $\nu_{p}(m-1)=0$, we can still apply Lemma \ref{modp} that is better in this case. Finally, we get
	\begin{align*}
		\binom{pm}{p}\equiv m \Mod{p^{2\nu_{p}(m)+\nu_{p}(m-1)+3-\chi}}.
\end{align*}
	However, we will not need this fact later. }
\end{rem}

Now we want to state a stronger version of Lemma \ref{parylem2}. Let us define
\begin{align*}
	\vartheta_{m}(n):=d_{m}(pn)-d_{m}(n).
\end{align*}

\begin{lem}\label{modp}
	Let $p\in\mathbb{P}_{\geq 3}$, $m\in\N_{\geq 1}$. Then $\vartheta_{m}(1)\equiv 0\Mod{p^{2\nu_{p}(m)+3-[p=3]}}$ and
	\begin{align*}
		\sum_{i=0}^{\min\{m,n\}}(-1)^{i}\binom{m}{i}\vartheta_{m}(n-i)\equiv (-1)^{n+1}\left[\binom{pm}{pn}-\binom{m}{n}\right] \Mod{p^{2\nu_{p}(m)+3-[p=3]}}.
	\end{align*}
\end{lem}
\begin{proof}
	We prove only the case $p\geq 5$. If $p=3$, the proof is analogous and in fact simpler because of some additional cancellations. Let us write $m=p^{\alpha}k$, where $p\nmid k$. For $n=1$ we have
	\begin{align*}
		d_{m}(p)= & -\sum_{i=1}^{p}(-1)^{i}\binom{pm}{i}d_{m}(p-i)+d_{m}(1)-md_{m}(0) \\
		 = & d_{m}(1)+\left[\binom{pm}{p}-m\right]d_{m}(0)-\sum_{i=1}^{p-1}(-1)^{i}\binom{pm}{i}d_{m}(p-i).
	\end{align*}
	By Lemma \ref{parylem4} we get that $p^{2\alpha +3}\mid \binom{pm}{p}-m$. Moreover,
	\begin{align*}
		\binom{pm}{i}=pm\frac{(pm-1)\ldots (pm-i+1)}{i!}\equiv pm\frac{(-1)^{i-1}}{i}\Mod{p^{\alpha +2}}
	\end{align*}
	and using Theorem \ref{parythm1} we obtain the congruence
	\begin{align*}
		\sum_{i=1}^{p-1}(-1)^{i}\binom{pm}{i}d_{m}(n-i)\equiv & \sum_{i-1}^{p-1}(-1)^{i}pm\frac{(-1)^{i-1}}{i}\frac{pm}{i}=-p^{2}m^{2}\sum_{i=1}^{p-1}\frac{1}{i^{2}}\equiv 0\Mod{p^{2\alpha +3}}.
	\end{align*}
	Thus always $d_{m}(p)\equiv d_{m}(1)\Mod{p^{2\alpha +3}}$, i.e., $\vartheta_{m}(1)\equiv 0\Mod{p^{2\alpha +3}}$. Now assume, that $n$ is arbitrary% and the statement is true for all numbers less that $n$
	. The recurrence relation for $d_{m}(n)$ gives
	\begin{align*}
		d_{m}(pn)= & -\sum_{i=1}^{\min\{pm,pn\}}(-1)^{i}\binom{pm}{i}d_{m}(pn-i)+\sum_{i=0}^{\min\{m,n\}}(-1)^{i}\binom{m}{i}d_{m}(n-i) \\
		 = & d_{m}(n) -\sum_{\substack{ i=1 \\ p\nmid i}}^{\min\{pm,pn\}}(-1)^{i}\binom{pm}{i}d_{m}(pn-i) \\
		 & +\sum_{i=1}^{\min\{m,n\}}(-1)^{i}\binom{m}{i}d_{m}(n-i) -\sum_{i=1}^{\min\{m,n\}}(-1)^{pi}\binom{pm}{pi}d_{m}(pn-pi) \\
		 = & d_{m}(n)+ \Sigma_{1}+\Sigma_{2},
	\end{align*}
	where
	\begin{align*}
		\Sigma_{1}= & -\sum_{\substack{ i=1 \\ p\nmid i}}^{\min\{pm,pn\}}(-1)^{i}\binom{pm}{i}d_{m}(pn-i), \\
		\Sigma_{2}= & \sum_{i=1}^{\min\{m,n\}}(-1)^{i}\binom{m}{i}d_{m}(n-i)-\sum_{i=1}^{\min\{m,n\}}(-1)^{i}\binom{pm}{pi}d_{m}(p(n-i)).
	\end{align*}
	At first, we deal with $\Sigma_{1}$. Observe, that if $i=\delta p+\gamma$ for some $\delta\in\N$ and $\gamma\in\{1,\ldots ,p-1\}$, then
	\begin{align*}
		\binom{pm}{i}= & pm\frac{(pm-1)\ldots (pm-i+1)}{i!}=pm\frac{\prod_{j=1}^{\delta}(pm-pj)}{\prod_{j=1}^{\delta}pj}\cdot\frac{\prod_{j=1,p\nmid j}^{i-1}(pm-j)}{\prod_{j=1,p\nmid j}^{i-1}j}\cdot \frac{1}{i} \\
		\equiv & pm\prod_{j=1}^{\delta}\frac{m-j}{j}\cdot\frac{\prod_{j=1,p\nmid j}^{i-1}(-j)}{\prod_{j=1,p\nmid j}^{i-1}j}\cdot \frac{1}{i}=pm\prod_{j=1}^{\delta}\frac{m-j}{j}\cdot(-1)^{i-1-\delta}\cdot\frac{1}{i} \Mod{p^{\alpha+2}}.
	\end{align*}
	Hence,
	\begin{align*}
		\Sigma_{1}\equiv & -\sum_{\substack{ i=1 \\ p\nmid i}}^{\min\{pm,pn\}}(-1)^{i}pm\prod_{j=1}^{\left\lfloor\frac{i}{p}\right\rfloor}\frac{m-j}{j}\cdot(-1)^{i-1-\left\lfloor\frac{i}{p}\right\rfloor}\cdot\frac{1}{i}\cdot\frac{pm}{-i} \\
		= & -p^{2}m^{2}\sum_{l=0}^{\min\{m-1,n-1\}}\sum_{r=1}^{p-1}(-1)^{l}\frac{1}{(pl+r)^{2}}\prod_{j=1}^{l}\frac{m-j}{j} \\
		= & -p^{2}m^{2}\sum_{l=0}^{\min\{m-1,n-1\}}(-1)^{l}\prod_{j=1}^{l}\frac{m-j}{j}\sum_{r=1}^{p-1}\frac{1}{(pl+r)^{2}} \\
		\equiv & -p^{2}m^{2}\sum_{l=0}^{\min\{m-1,n-1\}}(-1)^{l}\prod_{j=1}^{l}\frac{m-j}{j}\sum_{r=1}^{p-1}\frac{1}{r^{2}} \\
		\equiv & -p^{2}m^{2}\left(\sum_{r=1}^{p-1}r^{2}\right)\left(\sum_{l=0}^{\min\{m-1,n-1\}}(-1)^{l}\prod_{j=1}^{l}\frac{m-j}{j}\right) \\
		= & -p^{3}m^{2}\frac{(p-1)(2p-1)}{6}\left(\sum_{l=0}^{\min\{m-1,n-1\}}(-1)^{l}\prod_{j=1}^{l}\frac{m-j}{j}\right)\equiv 0\Mod{p^{2\alpha +3}}.
	\end{align*}
	Now we focus on $\Sigma_{2}$. We can write:
	\begin{align*}
		\Sigma_{2}= & \sum_{i=1}^{\min\{m,n\}}(-1)^{i}\binom{m}{i}\left[d_{m}(n-i)-d_{m}(p(n-i))\right]-\sum_{i=1}^{\min\{m,n\}}(-1)^{i}\left[\binom{pm}{pi}-\binom{m}{i}\right]d_{m}(p(n-i)) \\
		\equiv & -\sum_{i=1}^{\min\{m,n\}}(-1)^{i}\binom{m}{i}\vartheta_{m}(n-i)+(-1)^{n+1}\left[\binom{pm}{pn}-\binom{m}{n}\right] \Mod{p^{2\alpha +3}},
	\end{align*}
	because $p^{\alpha+3}\mid \binom{pm}{pn}-\binom{m}{n}$ by Lemma \ref{parylem4}, and $p^{\alpha+1}\mid d_{m}(p(n-i))$ for $i\neq n$. The result follows.
\end{proof}

\begin{rem}
{\rm From the above result we can write exact expression for the remainder of the division of $d_{m}(pn)-d_{m}(p)$ by $p^{2\nu_{p}(m)+3-[p=3]}$. More precisely, let
\begin{align*}
\Theta_{m}(x):=\sum_{n=0}^{\infty}\vartheta_{m}(n)x^{n}
\end{align*}
be the generating function of the sequence $(\vartheta_{m}(n))_{n=0}^{\infty}$. Lemma \ref{modp} can be restated in the following way:
\begin{align*}
	(1-x)^{m}\Theta_{m}(x)\equiv & \sum_{n=0}^{m}(-1)^{n+1}\left[\binom{pm}{pn}-\binom{m}{n}\right]x^{n} \\
	 \equiv & \sum_{n=0}^{m}(-1)^{n+1}\binom{pm}{pn}x^{n}+(1-x)^{m}\Mod{p^{2\nu_{p}(m)+3-[p=3]}}.
\end{align*}
Equivalently, due to the identity
\begin{align*}
\frac{1}{(1-x)^{m}}=\sum_{n=0}^{\infty}\binom{n+m-1}{n}x^{n},
\end{align*}
we get the congruence
\begin{align*}
\Theta_{m}(x)&\equiv\left(\sum_{n=0}^{\infty}\binom{n+m-1}{n}x^{n}\right)\left(\sum_{n=0}^{m}(-1)^{n+1}\left(\binom{pm}{pn}-\binom{m}{n}\right)x^{n}\right)\\
             &\equiv\sum_{n=0}^{\infty}\left(\sum_{k=0}^{\op{min}\{m,n\}}(-1)^{k+1}\left(\binom{pm}{pk}-\binom{m}{k}\right)\binom{n-k+m-1}{n-k}\right)x^{n}\Mod{p^{2\nu_{p}(m)+3-[p=3]}}.
\end{align*}

Observe, that we can also rewrite the above functional equation for $\Theta_{m}$ as
\begin{align*}
	-1+\Theta_{m}(x)\equiv & \frac{\sum_{n=0}^{m}(-1)^{n+1}\binom{pm}{pn}}{(1-x)^{m}} \\
	= & \sum_{n=0}^{\infty}\left(\sum_{k=0}^{\op{min}\{m,n\}}(-1)^{k+1}\binom{pm}{pk}\binom{n-k+m-1}{n-k}\right)x^{n} \Mod{p^{2\nu_{p}(m)+3-[p=3]}}.
\end{align*}

Thus simply:
\begin{equation*}%\label{dequiv}
	\begin{aligned}
		d_{m}(pn)-d_{m}(n)\equiv & \sum_{k=0}^{\op{min}\{m,n\}}(-1)^{k+1}\left(\binom{pm}{pk}-\binom{m}{k}\right)\binom{n-k+m-1}{n-k} \\
		\equiv & \sum_{k=0}^{\op{min}\{m,n\}}(-1)^{k+1}\binom{pm}{pk}\binom{n-k+m-1}{n-k} \Mod{p^{2\nu_{p}(m)+3-[p=3]}}
	\end{aligned}
\end{equation*}
for all $n\geq 1$.
}
\end{rem}

Lemma \ref{modp} allow us to simply get the following theorem concerning the problem of finding natural numbers $k$ such that the congruence $d_{m}(pn)\equiv d_{m}(n)\Mod{p^{\nu_{p}(m)+k}}$ holds for all $n\geq 1$.

\begin{thm}\label{parythm3}
	Let $m\in\N_{\geq 1}$ and $p\in\mathbb{P}_{\geq 3}$. If one of the following conditions holds:
	\begin{enumerate}
		\item $p\geq 5$,
		\item $p=3$ and $\nu_{3}(m)\geq 1$,
	\end{enumerate}
	then the congruence
	\begin{align*}
		d_{m}(pn)\equiv d_{m}(n)\Mod{p^{\nu_{p}(m)+3}}
	\end{align*}
	is true for all $n\geq 1$.	
\end{thm}
\begin{proof}
	It is easy to check, that if $p$ and $m$ satisfy one of the conditions from the statement then $2\nu_{p}(m)+3-[p=3]\geq \nu_{p}(m)+3$. Lemma \ref{modp} implies, that in order to prove the theorem, it is enough to show that $p^{\alpha +3}\mid \binom{pm}{pi}-\binom{m}{i}$ for $p$ and $m$ satisfying one of the conditions from the statement, and $i\in\{1,\ldots ,m\}$. If $p\geq 5$ it follows directly from Lemma \ref{parylem4}. Assume $p=3$ and $\nu_{3}(m)\geq 1$. From Lemma \ref{parylem4} it is enough to show that then $\nu_{p}\left(\binom{m}{i}\right)+3-\chi\geq 3$, i.e., $\nu_{p}\left(\binom{m}{i}\right)\geq \chi$. If $3\mid i$ this is obvious because $\chi =0$ then. If $3\nmid i$, then $3\mid \binom{m}{i}$, so $\nu_{p}\left(\binom{m}{i}\right)\geq 1=\chi$, so the inequality is also satisfied.
\end{proof}

\section{Questions, problems and conjectures}\label{sec6}

In this section we collect some questions, problems and conjectures which appeared during our work.

\bigskip

It is well known that if $k\in\N_{+}$ and $t\equiv 1 \Mod{2}$, then
\begin{align*}
c_{1}(2^{2k+1}t)-c_{1}(2^{2k-1}t)&\equiv 0\Mod{2^{3k+2}},\\
c_{1}(2^{2k}t)-c_{1}(2^{2k-2}t)&\equiv 0\Mod{2^{3k}}
\end{align*}
(remember $c_{1}(n)=b(2n)$, where $b(n)$ counts the binary partitions of $n$). The above congruences were conjectured by Churchhouse in \cite{Chu} and independently proved by R\"{o}dseth \cite{Rod} and Gupta \cite{Gup}. Moreover, it is known that there is no higher power of 2 which divides $c_{1}(4n)-c_{1}(n)$. These results generated a lot of research devoted to certain differences connected with other partition functions. This result motivates the question concerning the divisibility of the number $c_{m}(2^{k+2}n)-c_{m}(2^{k}n)$ by powers of 2. We performed some numerical computations in case of $m\in\{2,3,\ldots,10\}$ and $n\leq 10^5$ and believe that the following is true.

\begin{conj}
For $k\in\N_{+}$ and each $n\in\N_{+}$, we have:
\begin{align*}
\nu_{2}(c_{2k}(4n)-c_{2k}(n))=\nu_{2}(n)+2\nu_{2}(k)+3.
\end{align*}
Moreover, for $k\in\N$ and $n\in\N_{+}$ the following inequalities hold
\begin{align*}
\nu_{2}(c_{4k+1}(4n)-c_{4k+1}(n))&\geq \nu_{2}(n)+3,\\
\nu_{2}(c_{4k+3}(4n)-c_{4k+3}(n))&\geq \nu_{2}(n)+6.
\end{align*}
In each case the equality holds for infinitely many $n\in\N$.
\end{conj}

In case of $p\in\mathbb{P}_{\geq 3}$ we can use Theorem \ref{parythm3} to get some information about the behaviour of the differences $d_{m}(p^{2}n)-d_{m}(n)$ modulo powers of $p$. However, we predict that results obtained in this way are far from being optimal. For example, a direct application of Theorem \ref{parythm3} gives
\begin{align*}
	d_{m}(p^{2}n)-d_{m}(n)=\left(d_{m}(p^{2}n)-d_{m}(pn)\right)+\left(d_{m}(pn)-d_{m}(n)\right)\equiv 0\Mod{p^{\nu_{p}(m)+3}}.
\end{align*}
On the other hand, we believe that the following, much stronger property, is true.

\begin{conj}
Assume that $m=pk+i$ for some $k\in\N$ and $i\in\{0,\ldots ,p-1\}$. Then we have the following equality:
\begin{align*}
\nu_{p}(d_{m}(p^2n)-d_{m}(n))=\nu_{p}(n)+2\nu_{p}(m)+3-[p=3].
\end{align*}
In the case $i=p-1$ the following inequality is true:
\begin{equation*}
\nu_{p}(d_{pk+p-1}(p^2n)-d_{pk+p-1}(n))\geq \nu_{p}(n)+4-[p=3].
\end{equation*}
For infinitely many values of $n$ the above inequality is an equality.
\end{conj}

\bigskip

In the light of results obtained in the previous section it is quite natural to ask about existence of pairs $m\in\N_{+}$ and $k\in\N_{\geq 3}$ such that the following congruence
\begin{equation}\label{speccong}
d_{m}(pn)\equiv d_{m}(n)\Mod{p^{\nu_{2}(m)+k}}
\end{equation}
holds. We performed numerical search for values of $m\in\{2,\ldots,100\}$ such that the congruence (\ref{speccong}) is satisfied for $k=5$ or $k=6$, $p\in\{3, 5, \ldots, 29\}$ and all $n\leq 10^3$. For $k=5$ we found the following pairs
\begin{align*}
(p,m)= & (3,26), (3,27), (3,53), (3,54), (3,80), (3,81), (5,24), (5,25), (5,49), (5,50), \\
& (5,74), (5,75), (5,99), (5,100), (7,48), (7,49), (7,97), (7,98).
\end{align*}
Among the above pairs only the pairs $(p,m)=(3,80), (3,81)$ extend to solutions of (\ref{speccong}) with $k=6$.

We formulate the following general
\begin{prob}
For which triples $(p, m, k)\in\mathbb{P}_{\geq 3}\times \N_{+}\times \N_{\geq 3}$ the congruence {\rm (\ref{speccong})} holds for all $n\in\N_{+}$?
\end{prob}

\bigskip

Let $p\in\mathbb{P}$ and for $n\in\N_{+}$ write $n=\sum_{i=0}^{k}\eps_{i}p^{i}$, where $\eps_{i}\in\{0,1,\ldots,p-1\}$ and $k\leq \log_{p} n$. This representation is just the (unique) $p$-ary expansion of $n$ in base $p$. Let us observe that the equality $\nu_{p}(n)=u$ implies $\eps_{0}=\ldots=\eps_{u-1}=0$ and $\eps_{u}\neq 0$ in the above representation. Thus, if $m\in\Z\setminus\{-1\}$ is fixed, our results concerning the exact value of $\nu_{2}(c_{m}(n))$ and $\nu_{p}(d_{m}(n))$ for $p\in\mathbb{P}_{\geq 3}$ given in Theorem \ref{mainthm} and Theorems \ref{parythm1} and \ref{parythm2} respectively, imply that the set of values of the numbers of trailing zeros in the binary expansion of $c_{m}(n)$ and $p$-ary expansion of $d_{m}(n)), n\in\N_{+}$, is bounded. This observation suggests the question whether the index of the next non-zero digit in the binary expansion in $c_{m}(n)$ or $p$-ary expansion of $d_{m}(n)$ is in bounded distance from the first one. We state this in equivalent form as the following

\begin{ques}
Let $m\in\N_{+}$ and write
$$
u_{m}(n)=\nu_{2}\left(\frac{c_{m}(n)}{2^{\nu_{2}(c_{m}(n))}}-1\right),\quad v_{p,m}(n)=\nu_{p}\left(\frac{d_{m}(n)}{p^{\nu_{p}(d_{m}(n))}}-\left[\frac{d_{m}(n)}{p^{\nu_{p}(d_{m}(n))}}\Mod{p}\right]\right),
$$
where $p\in\mathbb{P}_{\geq 3}$. Does there exist $m\in\N_{+}$ such that some of the sequences
$(u_{m}(n))_{n\in\N_{+}}, (v_{p,m}(n))_{n\in\N_{+}}$ has finite set of values?
\end{ques}

%We performed some numerical experiments in the case of the sequence $(c_{m}(n))_{n\in\N}$. More precisely, let us write $u_{m}(n)=\nu_{2}\left(\frac{c_{m}(n)}{2^{\nu_{2}(c_{m}(n))}}-1\right)$.
In order to see what is going on, we performed numerical computations for $m\in\{1,\ldots, 100\}$ and $n\leq 10^5$. We observed that in the considered range there are many values of $m$ such that the set of values of the sequence $(u_{m}(n))_{n\in\N}$ is small (we looked for sets with cardinality $\leq 4$). We define:
$$
M_{m}(x):=\op{max}\{u_{m}(n):\;n\leq x\},\quad L_{m}(x):=|\{u_{m}(n):\;n\leq x\}|.
$$
In the table below we present the results of our computations.
\begin{equation*}
\begin{array}{|l|l|l|l|l|l|l|l|l|l|l|l|l|l|l|l|}
 \hline
  m          & 3 & 15 & 23 & 27 & 35 & 39 & 47 & 59 & 63 & 67 & 79 & 87 & 91 & 95 & 99\\
  \hline
  M_{m}(10^5)& 2 &  4 &  3 &  2 &  2 &  3 &  4 &  2 &  6 &  2 &  4 &  3 &  2 &  5 &  2 \\
  \hline
  L_{m}(10^5)& 2 &  3 &  3 &  2 &  2 &  3 &  3 &  2 &  4 &  2 &  3 &  3 &  2 &  3 &  2 \\
  \hline
\end{array}
\end{equation*}
\begin{center}
Table 1. Values of $m\in\{1,\ldots, 100\}$ such that $L_{m}(10^5)\leq 4$ together with the largest value of $u_{m}(n)$ for $n\leq 10^5$.
\end{center}

Our numerical computations strongly suggest that there should be infinitely many $m\in\Z$ such that the sequence $(u_{m}(n))_{n\in\N}$ is bounded. We even dare to formulate the following

\begin{conj}
Let $k\in\N_{+}$ and $m=2^{2k}-1$. Then the sequence $(u_{m}(n))_{n\in\N}$ is bounded.
\end{conj}

In fact, we expect that for $n\in\N$ the inequality $u_{2^{2k}-1}(n)\leq 2k$ is true.

In case of $p\in\mathbb{P}_{\geq 3}$ we expect that for all $m\in\N_{+}$ the sequence $(v_{p,m}(n))_{n\in\N_{+}}$ is unbounded.

\bigskip

Finally, let us note that in earlier sections we proved results concerning the $p$-adic behaviour of the sequence $(d_{m}(n))_{n\in\N}$ with $p\in\mathbb{P}$ and $m\in\Z\setminus\{0\}$. A first question which comes to mind is whether anything similar can be proved for the sequence $(S_{k,m}(n))_{n\in\N}$, where $k$ is not a prime number. Here we back to our general definition of $S_{k,m}(n)$ as the $n$-th coefficient in the power series expansion of $H_{k,m}(x)$. We thus ask about the behaviour of $\phi_{k}(S_{k,m}(n))$, where $\phi_{k}:\;\N\rightarrow \N$ is an analogue of the $p$-adic valuation function, i.e., for given $n\in\N$ we define
$$
\phi_{k}(n):=\max\{s\in\N:\;k^{s}\mid n\}\quad\mbox{and}\quad \phi_{k}(0):=+\infty.
$$
It is clear that if $k\not\in\mathbb{P}$, then the function $\phi_{k}$ is not additive. Indeed, if $k=k_{1}k_{2}$ with $k_{1}, k_{2}>1$ and $n_{1}=k_{1}u_{1}, n_{2}=k_{2}u_{2}$ with $\gcd(u_{1}u_{2},k)=1$, then $\phi_{k}(n_{1}n_{2})=1\neq \phi_{k}(n_{1})+\phi_{k}(n_{2})=0$. However, the question concerning the behaviour of the sequence $(\phi_{k}(S_{m,k}(n))_{n\in\N}$ is still interesting and non-trivial. Moreover, it seems that in the case of composite $k$ some new phenomena arise. We concentrate on the first non-trivial case, i.e., $k=4$, and formulate several conjectures.

\begin{conj}
\begin{itemize}
\item[(1)] If $m\equiv 6\Mod{8}$, then the sequence $(\phi_{4}(S_{4,m}(n)))_{n\in\N}$ is unbounded.
\item[(2)] If $m\not\equiv 6\Mod{8}$, then the sequence $(\phi_{4}(S_{4,m}(n)))_{n\in\N}$ is 4-automatic.
\item[(3)] If $a_{n}=\phi_{4}(S_{4,1}(n+1))$, then $a_{0}=a_{7}=1, a_{1}=a_{3}=0$ and
$$
a_{4n+2}=a_{4n},\quad a_{8n+5}=a_{8n+3}=a_{8n+1}=a_{4n+1},\quad a_{16n+7}=a_{2n},\quad a_{16n+15}=a_{n}.
$$
\item[(4)] If $b_{n}=\phi_{4}(S_{4,2}(n+1))$, then $b_{0}=b_{1}=1, b_{3}=2$ and
$$
b_{4n+2}=b_{4n+1}=b_{4n},\quad b_{8n+7}=b_{2n+1},\quad b_{16n+11}=b_{16n+3}=b_{8n+3}.
$$
\item[(5)] For $s\in\N_{\geq 2}$ we have
$$
\nu_{2}(S_{4,2^{s}}(n))=s+1+(\nu_{2}(2n)\Mod{2}).
$$
In particular,
$$
\phi_{4}(S_{4,2^{s}}(n))=\left\lfloor\frac{s+1}{2}\right\rfloor+\left((s+1)\nu_{2}(2n)\Mod{2}\right).
$$
\item[(6)] For $s\in\N_{\geq 3}$ and each $m\in\N$ we have
$$
\phi_{4}(S_{4,2^{s}m+2^{s-1}}(n))=\phi_{4}(S_{4,2^{s-1}}(n)).
$$
\end{itemize}
\end{conj}

Our numerical calculations suggest that for $k=p^{a}$ and $m=p^{b}, a, b\in\N_{+}$, the sequence of the $p$-adic valuations of $S_{k,m}(n)$ is $p$-automatic (and thus bounded). However, we were unable to formulate such nice formula for the corresponding valuation like in the case $k=4$ and $m=2^{s}$.

Anyway one can state the following general

\begin{prob}
For given $k=p^{a}$, where $p\in\mathbb{P}$ and $a\in\N_{+}$ characterize those values of $m\in\N_{+}$ such that the sequence $(\nu_{p}(S_{k,m}(n)))_{n\in\N}$ is $p$-automatic (is bounded).
\end{prob}

Numerical computations suggest the following

\begin{conj}
\begin{enumerate}
 \item[(1)] If $k$ is not a power of a prime number, then for each $m\in\N_{+}$ the sequence $(\phi_{k}(S_{k,m}(n)))_{n\in\N}$ is unbounded.
 \item[(2)] If $k$ is not a power of a prime number, then for each $m\in\N_{+}$ the sequence $(\phi_{k}(S_{k,m}(n)))_{n\in\N}$ is not $k$-regular.
\end{enumerate}
\end{conj}

%We also believe that

\vskip 1cm

\noindent Maciej Ulas, Jagiellonian University, Faculty of Mathematics and Computer Science, Institute of
Mathematics, {\L}ojasiewicza 6, 30-348 Krak\'ow, Poland; email:
maciej.ulas@uj.edu.pl

\bigskip

\noindent B{\l}a\.{z}ej \.{Z}mija, Jagiellonian University, Faculty of Mathematics and Computer Science, Institute of
Mathematics, {\L}ojasiewicza 6, 30-348 Krak\'ow, Poland; email:
blazej.zmija@im.uj.edu.pl

 \end{document}